\documentclass[11pt,reqno,tbtags]{amsart}

\usepackage{graphicx}
\setkeys{Gin}{width=\linewidth,totalheight=\textheight,keepaspectratio}
\graphicspath{{graphics/}}

\usepackage{units}

\usepackage{amsthm,amssymb,amsfonts,amscd,amssymb,eucal,latexsym}
\usepackage[numbers, sort&compress]{natbib}
\usepackage{mathrsfs}
\usepackage{graphicx}
\usepackage{paralist}
\usepackage[pagebackref=true]{hyperref}
\usepackage{color}
\usepackage[normalem]{ulem}
\usepackage{stmaryrd} 
\usepackage{enumitem} 
\usepackage{caption,subcaption}

\hypersetup{ colorlinks=true, linkcolor=red, urlcolor=blue, citecolor=red}

\providecommand{\R}{}

\providecommand{\N}{}

\renewcommand{\R}{\mathbb{R}}

\renewcommand{\N}{{\mathbb N}}

\newcommand{\E}[1]{{\mathbf E}\left[#1\right]}
\newcommand{\e}{{\mathbf E}}

\newcommand{\p}[1]{{\mathbf P}\left\{#1\right\}}
\providecommand{\P}{}
\renewcommand{\P}[1]{{\mathbf P}\left\{#1\right\}}

\newcommand{\I}[1]{{\mathbf 1}_{[#1]}}
\newcommand{\set}[1]{\left\{ #1 \right\}}

\newcommand{\Cexp}[2]{\mathbf{E}\set{\left. #1 \; \right| \; #2}}

 \newcommand{\bag}{\begin{align}}
\newcommand{\bags}{\begin{align*}}
\newcommand{\eag}{\end{align*}}
\newcommand{\eags}{\end{align*}}

\newtheorem{thm}{Theorem}[section]
\newtheorem{lem}[thm]{Lemma}
\newtheorem{prop}[thm]{Proposition}
\newtheorem{cor}[thm]{Corollary}

\newtheorem{ex}{Exercise}[section]
\newtheorem{sex}[ex]{Exercise}
\newtheorem{fact}[thm]{Fact}

\numberwithin{equation}{section}

\newcommand\cC{\mathcal C}

\newcommand{\pran}[1]{\left(#1\right)}

\providecommand{\eps}{}
\renewcommand{\eps}{\epsilon}
\providecommand{\ora}[1]{}
\renewcommand{\ora}[1]{\overrightarrow{#1}}

\definecolor{clou}{rgb}{0.8,0.25,0.5125}

\newcommand\urladdrx[1]{{\urladdr{\def~{{\tiny$\sim$}}#1}}}

\begingroup
  \divide\count255 by 60
  \multiply\count255 by -60
  \advance\count255 by \time
  \ifnum \count255 < 10 \xdef\oclock{\the\count1:0\the\count255}
  \else\xdef\oclock{\the\count1:\the\count255}\fi
\endgroup

\DeclareRobustCommand{\SkipTocEntry}[5]{}

\usepackage{etoolbox} 

\usepackage{framed,mdframed,color}
\definecolor{shadecolor}{rgb}{0.9,0.9,0.9}
\definecolor{mygray}{gray}{0.97}
\definecolor{redborder}{rgb}{0.4,0,0}
\definecolor{boxbackground}{gray}{0.9}
\newcommand{\bex}{\begin{leftbar} \begin{ex}}
\newcommand{\bsex}{\begin{leftbar} \begin{sex} {\large \textrm{$\circledast$~}}}
\newcommand{\eex}{\end{ex} \end{leftbar}}
\newcommand{\esex}{\end{sex} \end{leftbar}}
\renewcommand{\[}{\begin{equation*}}
\renewcommand{\]}{\end{equation*}}

\providecommand{\q}{}
\renewcommand{\q}{\alpha} 

\usepackage{natbib}

\newtoggle{bookversion}
\togglefalse{bookversion}

\newcommand{\chsec}[1]{ 
\iftoggle{bookversion}{%
  \chapter{#1}%
}{%
  \section{#1}%
}%
}
\newcommand{\secsub}[1]{ 
\iftoggle{bookversion}{%
  \section{#1}
}{%
  \subsection{#1}
}}
\newcommand{\subpar}[1]{ 
\iftoggle{bookversion}{%
  \subsection{#1}
}{%
\medskip
  \subsubsection{{#1}}
}}

\newcommand{\isfullwidth}[1]{
\iftoggle{bookversion}{
\begin{fullwidth}
#1
\end{fullwidth}
}{
#1
}}
\newenvironment{marginfig}{
\iftoggle{bookversion}{\begin{marginfigure}}{\begin{figure}}
}{\iftoggle{bookversion}{\end{marginfigure}}{\end{figure}}}

	{
	\par\noindent\textbf{Exercises}\begin{itshape}\par\noindent
	  \begin{list}{\arabic{chapter}.\arabic{section}.\arabic{subsection}}
	    {\usecounter{subsection}
	    \setlength{\rightmargin}{\leftmargin}}
	}
	 {
	 \end{list}
	 \end{itshape}
	 }

\mdfdefinestyle{algorithm}{leftmargin=0cm,rightmargin=0cm,innerleftmargin=0.3cm,innerrightmargin=1cm,roundcorner=10pt}
\mdfapptodefinestyle{algorithm}{backgroundcolor=boxbackground,linecolor=redborder,linewidth=3pt}

\begin{document}

\title[Discrete coalescents]{Partition functions of discrete coalescents: from Cayley's formula to Frieze's $\zeta(3)$ limit theorem} 
\author{Louigi Addario-Berry}
\address{Department of Mathematics and Statistics, McGill University, 805 Sherbrooke Street West, 
		Montr\'eal, Qu\'ebec, H3A 2K6, Canada}
\email{louigi@math.mcgill.ca}
\date{July 31, 2014} 
\urladdrx{http://www.math.mcgill.ca/~louigi/}


\begin{abstract} 
In these expository notes, we describe some features of the multiplicative coalescent and its connection with random graphs and minimum spanning trees. We use Pitman's proof \citep{pitman99coalescent} of Cayley's formula, which proceeds via a calculation of the partition function of the additive coalescent, as motivation and as a launchpad. We define a random variable which may reasonably be called the empirical partition function of the multiplicative coalescent, and show that its typical value is exponentially smaller than its expected value. Our arguments lead us to an analysis of the susceptibility of the Erd\H{o}s-R\'enyi random graph process, and thence to a novel proof of Frieze's $\zeta(3)$-limit theorem for the weight of a random minimum spanning tree. 
\end{abstract}

\maketitle

\chsec{Introduction}\label{sec:intro}
Consider a discrete time process $(P_i,1 \le i \le n)$ of coalescing blocks, with the following dynamics. The process starts from the partition of $[n]=\{1,\ldots,n\}$ into singletons: $P_1 = \{\{1\},\ldots,\{n\}\}$. To form $P_{i+1}$ from $P_i$ choose two parts $P,P'$ from $P_i$ and merge them. 
We assume there is a function $\kappa$ such that the probability of choosing parts $P,P'$ is proportional to $\kappa(|P|,|P'|)$; call $\kappa$ a {\em gelation kernel}. 

Different gelation kernels lead to different dynamics. 
Three kernels whose dynamics have been studied in detail are $\kappa(x,y)=1$, $\kappa(x,y)=x+y$, and $\kappa(x,y)=xy$; these are often called Kingman's coalescent, the additive coalescent, and the multiplicative coalescent, respectively. In these cases there is a natural way to enrich the process and obtain a {\em forest-valued} coalescent. 

These notes are primarily focussed on the properties of the forest-valued multiplicative coalescent. We proceed from a statistical physics perspective, and begin by analyzing the partition functions of the three coalescents. Here is what we mean by this. Say that a sequence $(P_1,\ldots,P_n)$ of partitions of $[n]$ is an {\em $n$-chain} if $P_1=\{\{1\},\ldots,\{n\}\}$ is the partition of $n$ into singletons, and for $1 \le i < n$, $P_{i+1}$ can be formed from $P_i$ by merging two parts of $P_i$.
Think of $\kappa(x,y)$ as the number of possible ways to merge a block of size $x$ with one of size $y$. Then corresponding to an $n$-chain $P=(P_1,\ldots,P_n)$ there are 
\[
\prod_{i=1}^{n-1} \kappa(|A_i(P)|,|B_i(P)|)\, 
\]
possible ways that the coalescent may have unfolded; here we write $A_i(P)$ and $B_i(P)$ for the blocks of $P_i$ that are merged in $P_{i+1}$. 
Writing $\mathcal{P}_n$ for the set of $n$-chains, it follows that the total number of possibilities for the coalescent with gelation kernel $\kappa$ is 
\[
\sum_{P=(P_1,\ldots,P_n) \in \mathcal{P}_n} \prod_{i=1}^{n-1} \kappa\pran{|A_i(P)|,|B_i(P)|}\, ,
\]
and we view this quantity as the partition function of the coalescent with kernel $\kappa$. 

The partition functions of Kingman's coalescent and the additive and multiplicative coalescents have particularly simple forms: they are 
\begin{align*}
Z_{\textsc{kc}}(n)	& =n!(n-1)!\, ,\\
Z_{\textsc{ac}}(n)	& = n^{n-1}(n-1)!\, , \mbox{ and}\\
 Z_{\textsc{mc}}(n)	& =n^{n-2}(n-1)!\, .
 \end{align*}

These formulae are proved in Section~\ref{sec:threecoalescents}. A corollary of the formula for $Z_{\textsc{kc}}(n)$ is that the number of increasing trees with $n$ vertices is $(n-1)!$; this easy fact is well-known. The formula for $Z_{\textsc{ac}}(n)$ is due to Pitman \cite{pitman99coalescent}, who used it to give a beautiful proof of Cayley's formula; this is further detailed in Section~\ref{sec:cayleysformula}. 

It may seem surprising that the partition function of the multiplicative coalescent is so similar to that of the additive coalescent: near start of the process, when most blocks have size $1$, the additive coalescent has twice as many choices as the multiplicative coalescent. Later in the process, blocks should be larger, and one would guess that usually $xy > x+y$. Why these two effects should almost exactly cancel each other out is something of a mystery. On the other hand, the similarity of the partition functions may suggest that the additive and multiplicative coalescents have similar behaviour. 

A more detailed investigation will reveal interesting behaviour whose subtleties are not captured by the above formulae. 
We will see in Section~\ref{intro:mult} that there is a naturally defined ``empirical partition function'' $\hat{Z}_{\textsc{mc}}(n)$ such that $Z_{\textsc{mc}}(n)=\E{\hat{Z}_{\textsc{mc}}(n)}$. However, $\hat{Z}_{\textsc{mc}}(n)$ is typically {\em exponentially smaller} than $Z_{\textsc{mc}}(n)$ (see Corollary~\ref{cor:exp_decay}), so in a quantifiable sense, the partition function $Z_{\textsc{mc}}(n)$ takes the value it does due to extremely rare events. Correspondingly, it turns out that the behaviour of the additive and multiplicative coalescents are typically quite different. 

To analyze the typical value of $\hat{Z}_{\textsc{mc}}(n)$, we are led to develop the connection between the multiplicative coalescent and the classical Erd\H{o}s-R\'enyi random graph process $(G(n,p),0 \le p \le 1)$. The most technical part of the notes is the proof of a concentration result for the susceptibility of $G(n,p)$; this is Theorem~\ref{thm:susc_conc}, below. Using a well-known coupling between the multiplicative coalescent and Kruskal's algorithm for the minimum weight spanning tree problem, our susceptibility bound leads easily to a novel proof of the $\zeta(3)$ limit for the total weight of the minimum spanning tree of the complete graph (this is stated in Theorem~\ref{thm:frieze}, below).\footnote{We find this proof of the $\zeta(3)$ limit for the MST weight pleasing, as it avoids lemmas which involve estimating the number of unicyclic and complex components in $G(n,p)$; morally, the cycle structure of components of $G(n,p)$ should be unimportant, since cycles are never created in Kruskal's algorithm!}

\subsection*{Stylistic remarks}
\addtocontents{toc}{\SkipTocEntry}
The primary purpose of these notes is expository (though there are some new results, notably Theorems~\ref{thm:zmc_lower} and~\ref{thm:susc_conc}). Accordingly, we have often opted for repitition over concision. We have also included plenty of exercises and open problems (the open problems are mostly listed in Section~\ref{sec:openproblems}). Some exercises state facts which are required later in the text; these are distinguished by a {\large \textrm{$\circledast$}}. 

{
\hypersetup{linkcolor=blue}
\tableofcontents
}

\chsec{A tale of three coalescents}\label{sec:threecoalescents}
\secsub{Cayley's formula and Pitman's coalescent} \label{sec:cayleysformula}

We begin by describing the beautiful proof of 
Cayley's formula found by Jim Pitman, and its link with uniform spanning trees. Cayley's formula states that the number of trees with vertices $\{1,2,\ldots,n\}$ is $n^{n-2}$, or equivalently that the number of {\em rooted} trees with vertices labeled by $[n]:=\{1,2,\ldots,n\}$ is $n^{n-1}$. To prove this formula, \citet{pitman99coalescent} analyzes a process  we call Pitman's coalescent. To explain the process, we need some basic definitions. A {\em forest} is a graph with no cycles; its connected components are its trees. A {\em rooted forest} is a forest in which each tree $t$ has a distinguished root vertex $r(t)$. 

\medskip
\isfullwidth{
\begin{mdframed}[style=algorithm]
{\bf Pitman's Coalescent, Version 1.}
The process has $n$ steps, and at step $i$ consists of a rooted forest $F_i=\{T_1^{(i)},\ldots,T_{n+1-i}^{(i)}\}$ with $n+1-i$ trees. (At step $1$, these trees are simply isolated vertices with labels $1,\ldots,n$.) 
To obtain $F_{i+1}$ from $F_i$, choose a pair $(U_i,V_i)$, where $U_i \in [n]$ and $V_i$ is the root of some tree of $F_i$ not containing $U_i$, uniformly at random from among all such pairs. 
 Add an edge from $U_i$ to $V_i$, and root the resulting tree at the root of $U_i$'s old tree. The forest $F_{i+1}$ consists of this new tree together with the $n-i-1$ unaltered trees from $F_i$. 
\end{mdframed}
}
\medskip
The coalescents we consider all have the general form of Pitman's coalescent: they are forest-valued stochastic processes $(F_i,1 \le i \le n)$, where $F_i=\{T^{(i)}_1,\ldots,T^{(i)}_{n+1-i}\}$ is a forest with vertices labeled by $[n]$. 
\medskip
\isfullwidth{
\begin{mdframed}[style=algorithm]
{\bf Pitman's Coalescent, Version 2.} Consider the directed graph $K_n^{\to}$ with vertices $\{1,\ldots,n\}$ and an oriented edge from $k$ to $\ell$ for each $1 \le k \ne \ell \le n$. Let ${\mathbf{W}}=\{W_{(k,\ell)}: 1 \le k \ne \ell \le n\}$ be independent copies of a continuous random variable $W$, that weight the edges of $K_n^{\to}$. Let $F_1$ be as in Version $1$. For $i \in \{1,\ldots,n-1\}$, form $F_{i+1}$ from $F_i$ by adding the smallest weight edge $(k,\ell)$ whose head $k$ is the root of one of the trees in $F_i$. (Each tree of $F_i$ is rooted at its unique vertex having indegree zero in $F_i$.) 
\end{mdframed}
}
\medskip

Note that in Version 2, for each $i \in \{1,\ldots,n\}$ and each tree $T$ of $F_i$, all edges of $T$ are oriented away from a single vertex of $T$; so, viewing this vertex as the root of $T$, the orientation of edges in $T$ is fully specified by the location of its root. 
\begin{leftbar}
\begin{ex}\label{ex:pitman_cayley}
View the trees of Version 2 as rooted rather than oriented. Then the sequences of forests $(F_1,\ldots,F_n)$ described in Version 1 and Version 2 have the same distribution. 
\end{ex}
\end{leftbar}

Say that a finite set $\{X_i,i \in I\}$ of random variables is {\em exchangeable} if for any two deterministic orderings of $I$ as, say, $i_1,\ldots,i_k$ and $i_1',\ldots,i_i'$, the vectors $(X_{i_1},\ldots,X_{i_k})$ and $(X_{i_1'},\ldots,X_{i_k'})$ are identically distributed. In particular, if the elements of $\{X_i,i \in I\}$ are iid then the set is exchangeable. 
\begin{leftbar}
\begin{ex}\label{ex:pitman_cayley_2}
Suppose that the edge weights $\mathbf{W}$ are only assumed to be exchangeable and a.s.\ pairwise distinct. 
Show that the sequences of forests $(F_1,\ldots,F_n)$ described in Version 1 and Version 2 still have the same distribution. 
\end{ex}
\end{leftbar}

\begin{marginfig}%
\iftoggle{bookversion}{
  \includegraphics[width=1.2\linewidth,angle=90]{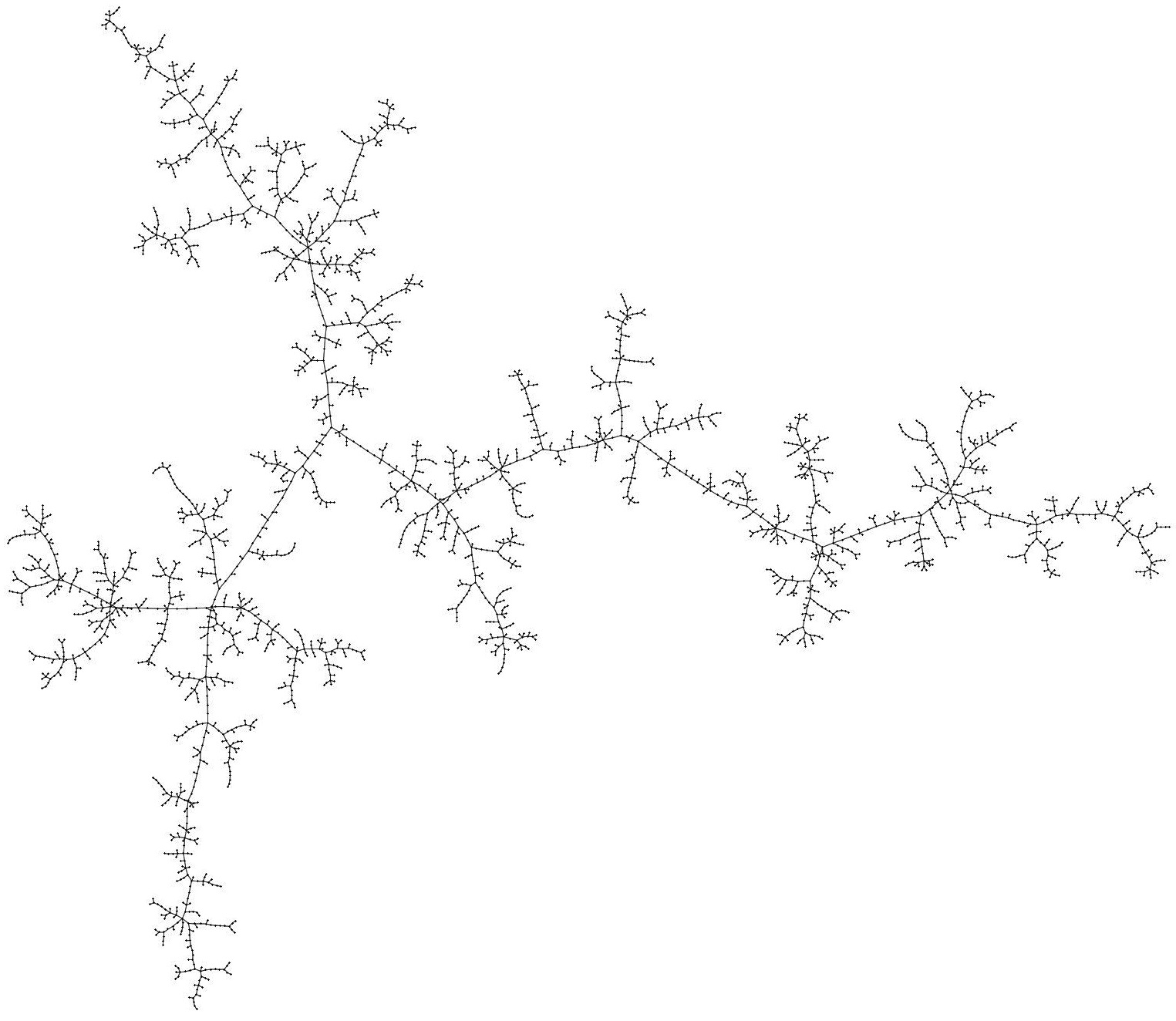}
  }
  {
  \includegraphics[width=0.7\linewidth,angle=90]{ust}
  }
  \caption{One of the $3000^{2998}$ labeled trees with $3000$ vertices, selected uniformly at random.}
  \label{mfig:additive_coal_tree}
\end{marginfig}

To prove Cayley's formula, we compute the {\em partition function} of Pitman's coalescent: this is the total number of possibilities for its execution. (To do so, it's easiest to think about Version 1 of the procedure.) For example, when $n=3$, there are $6$ possibilities for the first step of the process: $3$ choices for the first vertex, then $2$ choices of a tree not containing the first vertex. For the second step, there are $3$ choices for the first vertex; there is only one component not containing the chosen vertex, and we must choose it. Thus, for $n=3$, the partition function has value $Z_{\textsc{ac}}(3)=6\cdot3=18$. More generally, for the $n$-vertex process, when adding the $i$'th edge we have $n$ choices for the first vertex and $n-i$ choices of tree not containing the first vertex, so a total of $n(n-i)$ possibilities. Thus the partition function is 
\begin{equation}\label{eq:additive_coal_partition}
Z_{\textsc{ac}}(n) = \prod_{i=1}^{n-1} n \cdot (n-i) = n^{n-1}(n-1)!
\end{equation}

It is not possible to recover the entire execution path of the additive coalescent from the final tree, since there is no way to tell in which order the edges were added. If we wish to retain this information, we may label each edge of $T_1^{(n)}$ with the step at which it was added. More precisely, $L(e)$ is the unique integer $i \in \{1,\ldots,n-1\}$ such that $e$ is not an edge of $F_i$ but is an edge of $F_{i+1}$. It follows from the definition of the process that the edge labels are distinct, so $L:E(T_1^{(n)}) \to \{1,\ldots,n-1\}$ is a bijective map. 

Now fix a rooted tree $t$ with vertices $\{1,\ldots,n\}$, and consider the {\em restricted} partition function 
$Z_{\textsc{ac},t}(n)$; this is simply the number of possibilities for the execution of the process for which the end result is the tree $t$. We claim that $Z_{\textsc{ac},t}(n)=(n-1)!$.  This is easy to see: for any labelling $\ell$ of the edges of $t$ with integers $\{1,\ldots,n-1\}$, there is a unique execution path for which $(T_1^{(n)},L)=(t,\ell)$, and there are $(n-1)!$ possible labellings $\ell$. 
Thus, the probability of ending with the tree $t$ is $Z_{\textsc{ac},t}(n)/Z_{\textsc{ac}}(n) = 1/n^{n-1}$. Since this number doesn't depend on $t$, only on $n$, it follows that every rooted labelled tree with $n$ vertices is equally likely, and so there must be $n^{n-1}$ such trees. 

\medskip
\isfullwidth{
{
{\bf Note.} {\em The preceding argument is correct, but treads lightly around an important point. When performing the process, the number of possibilities for the $i$'th edge does not depend on the first $i-1$ choices, so the probability of building a particular tree $t$ by adding its edges in in a particular order is $[n^{n-1}(n-1)!]^{-1}$ regardless of the order. Of course, the {\em set} of possible choices at a given step must depend on the history of the process -- for example, we must not add a single edge twice. More generally, thinking of Version 2, applying the procedure to a graph other than $K_n^{\to}$ need not yield a uniform spanning tree of the graph, and indeed may not even build a tree. (Consider, for example, applying the procedure to a two-edge path.)}}
}
\medskip

By stopping Pitman's coalescent before the end, one can use a similar analysis to obtain counting formulae for forests. Write $Z_{\textsc{ac}}(n,k)$ for the total number of possibilities for Pitman's coalescent stopped at step $k$ (so ending with $n+1-k$ forests). We write $(m)_{\ell}$ to denote the {\em falling factorial} $\prod_{i=0}^{\ell-1} (m-i)$. 
\bex \label{ex:ac_partial_partitition}
\begin{itemize}
\item[(a)] Show that $Z_{\textsc{ac}}(n,k) = n^{k-1}(n-1)_{k-1}$ for each for $1 \le k \le n$. 
\item[(b)] An {\em ordered labeled forest} is a sequence $(t_1,\ldots,t_\ell)$ where each $t_i$ is a rooted labelled tree and all labels of vertices in the forest are distinct. Show that for each $1 \le \ell \le n$ the number of ordered labeled forests $(t_1,\ldots,t_\ell)$ with $\bigcup_{i=1}^\ell V(t_i)=[n]$, is $\ell\cdot (n)_\ell \cdot n^{n-\ell-1}$. 
\end{itemize}
\eex

We briefly discuss a special case of Version 2. Suppose that $W_{(k,l)}$ is exponential with rate $X_{(k,\ell)}$, where $\mathbf{X}=\{X_{(k,\ell)}: 1 \le k \ne \ell \le n\}$ are independent copies of any non-negative random variable $X$. By standard properties of exponentials and the symmetry of the process, the dynamics in this case may be described as follows. 

\medskip
\isfullwidth{
\begin{mdframed}[style=algorithm]
{\bf Pitman's Coalescent, Version 3.} Let $F_i$ be as in Version $1$. For $i \in \{1,\ldots,n-1\}$, choose an edge whose head is the root of any one of the trees in $F_i$, each such edge $(k,l)$ chosen with probability proportional to its weight $X_{(k,l)}$; add the chosen edge to create the forest $F_{i+1}$. 
\end{mdframed}
}
\medskip

Consider Version 3 of the procedure after $i-1$ edges have been added. 
Conditional on $\mathbf{X}$ and on the forest $(T_1^{(i)},\ldots,T_{n-i+1}^{(i)})$, the probability of adding a particular edge $(k,\ell)$ whose tail is a root, is proportional to $X_{(k,\ell)}$, so is equal to 
\[ 
\frac{X_{(k,\ell)}}{\sum_{m=1}^{n-i+1}\sum_{j \in \{1,\ldots,n\}\setminus V(T^{(i)}_m)}X_{(r(T_m^{(i)}),j)}}
 \, .
\]
Now fix any sequence $f_1,\ldots,f_n$ of forests that can arise in the process. Write $f_i = (t^{(i)}_{k},1 \le k \le n+1-i)$ and for $i=1,\ldots,n-1$ write $(k_i,\ell_i)$ for the unique edge of $f_{i+1}$ not in $f_{i}$. Then by the above,  
\begin{equation*}
\P{F_i=f_i,1 \le i \le n ~|~ \mathbf{X}} = \prod_{i=1}^{n-1} \frac{X_{(k_i,\ell_i)}}{\sum_{m=1}^{n-i+1}\sum_{j \in \{1,\ldots,n\}\setminus V(t^{(i)}_m)}X_{(r(t_m^{(i)}),j)}}\, .
\end{equation*}
By Exercise~\ref{ex:pitman_cayley} and the above analysis, it follows that for any such sequence $f_1,\ldots,f_n$, 
\[
\E{\prod_{i=1}^{n-1} \frac{X_{(k_i,\ell_i)}}{\sum_{m=1}^{n-i+1}\sum_{j \in \{1,\ldots,n\}\setminus V(t^{(i)}_m)}X_{(r(t_m^{(i)}),j)}}} = \frac{1}{n^{n-1}(n-1)!}\, .
\]
It is by no means obvious at first glance that this expectation should be not depend on law of $X$, let alone that it should have such a simple form.  

\secsub{Kingman's coalescent and random recursive trees}

Pitman's coalescent starts from isolated vertices labeled from $\{1,\ldots,n\}$, and builds a rooted tree by successive edge addition. At each step, an edge is added {\em to} some vertex, {\em from} some root (of a component not containing the chosen vertex). 
\iftoggle{bookversion}{We straightforwardly saw that at the stage when there are $i$ trees, the number of of possibilities for the next edge was $n(i-1)$. (It was important that this quantity depended only on the {\em number} of trees and not, say, their sizes, or some other feature.)}{When we calculated $Z_{\textsc{ac}}(n)$, it was important that the number of possibilities at each step depended only on the {\em number} of trees in the current forest and not, say, their sizes, or some other feature.}

Pitman's merging rule ({\em to} any vertex, {\em from} a root) yielded a beautiful proof of Cayley's formula. It is natural to ask what other rules exist, and what information may be gleaned from them. Of course, {\em from} any vertex, {\em to} a root just yields the additive coalescent, with edges of the resulting tree oriented towards the root rather than towards the leaves. What about {\em from} any root, {\em to} any (other) root, as in the following procedure? In a very slight abuse of terminology, we call this rule {\em Kingman's coalescent}. We again start from a rooted forest $F_1$ of $n$ isolated vertices $\{1,\ldots,n\}$. Recall that we write $F_i=\{T_1^{(i)},\ldots,T_{n+1-i}^{(i)}\}$. 

\medskip
\isfullwidth{
\begin{mdframed}[style=algorithm]
{\bf Kingman's Coalescent.}
At step $i$, choose an ordered pair $(U_i,V_i)$ of distinct roots from $\{r(T_1^{(i)}),\ldots,r(T_{n+1-i}^{(i)}\}$, uniformly at random from among the $(n+1-i)(n-i)$ such pairs. Add an edge from $U_i$ to $V_i$, and root the resulting tree at $U_i$. The forest $F_{i+1}$ consists of this new tree together with the $n-i-1$ unaltered trees from $F_i$. 
\end{mdframed}
}
\medskip

Our convention is that when an edge is added from $u$ to $v$, the root of the resulting tree is $u$; this maintains that edges are always oriented towards the leaves. 
For Kingman's coalescent, when $i$ trees remain there are $i(i-1)$ possibilities for which oriented edge to add. Like for Pitman's coalescent, this number depends only on the number of trees, and it follows that the total number of possible execution paths for the process is 
\begin{equation}\label{eq:mult_coal_partition}
Z_{\textsc{kc}}(n) = 
\prod_{i=2}^n i(i-1) = n!(n-1)!\, .
\end{equation}
What does this number count? 

To answer the preceding question, as in the additive coalescent let $L:E(T_1^{(n)})\to \{1,\ldots,n-1\}$ label the edges of $T_1^{(n)}$ in their order of addition. It is easily seen that for Kingman's coalescent, the edge labels decrease along any root-to-leaf path; we call such a labelling a {\em decreasing edge labelling}.\footnote{It is more common to order by {\em reverse} order of addition, so that labels increase along root-to-leaf paths; this change of perspective may help with Exercise~\ref{ex:random_recursive_trees}.} 
Furthermore, any decreasing edge labelling of $T_1^{(n)}$ can arise. Once again, the full behaviour of the coalescent is described by pair $(T_1^{(n)},L)$, and conversely, the coalescent determines $T_1^{(n)}$ and $L$. These observations yield that the number of rooted trees with vertices labelled $\{1,\ldots,n\}$, additionally equipped with a decreasing edge labelling, is $n!(n-1)!$. The factor $n!$ simply counts the number of ways to assign the labels $\{1,\ldots,n\}$ to the vertices. By symmetry, each vertex labelling of a given tree is equally likely to arise, and so we have the following. 

\begin{prop}\label{prop:increasing_tree_count}
The number of pairs $(T,L)$, where $T$ is a rooted tree with $n$ vertices and $L$ is a decreasing edge labelling of $T$, is $(n-1)!$. 
\end{prop}
\bex[Random recursive trees] \label{ex:random_recursive_trees}
Prove Proposition~\ref{prop:increasing_tree_count} by introducing and analyzing an $n$-step procedure that at step $i$ consists of a rooted tree with $i$ vertices.
\eex
\begin{marginfig}%
\iftoggle{bookversion}{
  \includegraphics[width=\linewidth,angle=90]{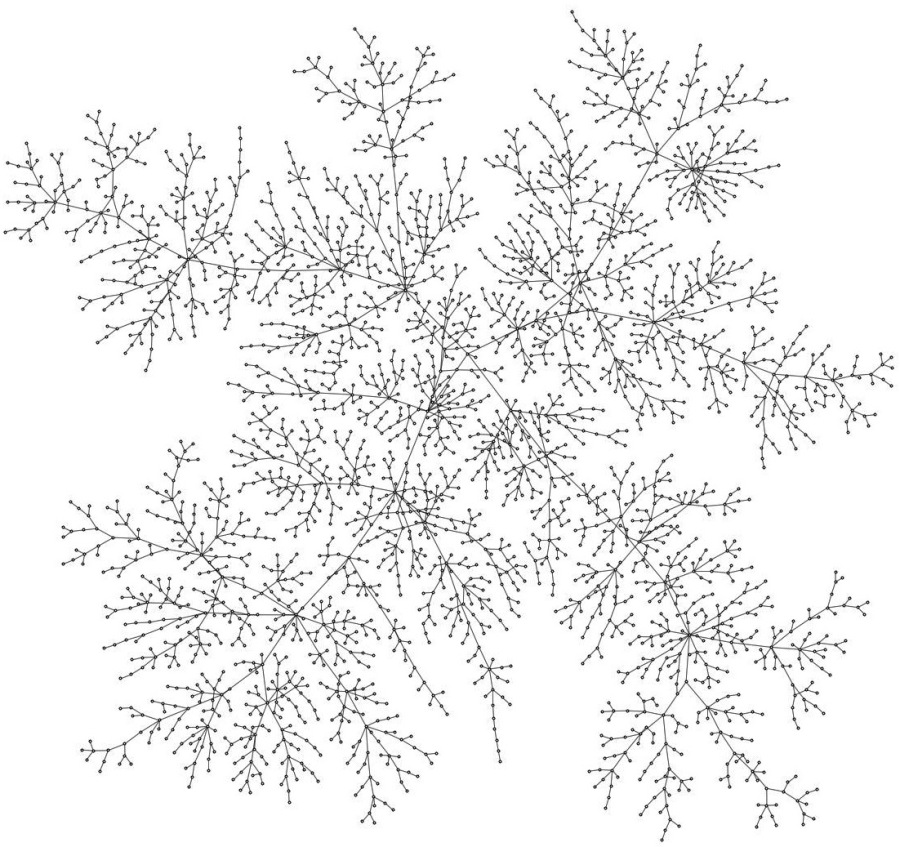}
  }
  {
    \includegraphics[width=0.6\linewidth,angle=90]{randomrecursivetree}
    }
  \caption{One of the $2999!$ rooted trees on $3000$ vertices with a decreasing edge labelling (labels suppressed).}
  \label{mfig:kingman_coal_tree}
\end{marginfig}
Before the next exercise, we state a few definitions. For a graph $G$, write $|G|$ for the number of vertices of $G$. 
If $T$ is a rooted tree and $u$ is a vertex of $T$, write $T_u$ for the subtree of $T$ consisting of $u$ together with its descendants in $T$ (we call $T_u$ the subtree of $T$ rooted at $u$). Also, if $u$ is not the root, write $p(u)$ for the parent of $u$ in $T$. 
\bex
Show that for a fixed rooted tree $T$, the number of decreasing edge labellings of $T$ is 
\[
\prod_{v \in V(T)} \frac{(|T_v|-1)!}{\prod_{\{u \in V(T): p(u)=v\}} |T_u|!}\, . 
\]
\eex
Our convention is that an empty product equals $1$; a special case is that $0!=1$. 
It follows from the preceding exercise that, writing $\mathcal{T}_n$ for the set of rooted trees with $n$ vertices, 
\[
\sum_{T \in \mathcal{T}_n} \prod_{v \in V(T)} \frac{|E(T_u)|!}{\prod_{\{u \in V(T): p(u)=v\}} |V(T_u)|!} = (n-1)!\, ; 
\]
again, a formula that is far from obvious at first glance!. 

To finish the section, note that just like for Pitman's coalescent, we might well consider a version of this procedure that is ``driven by'' iid non-negative weights ${\mathbf{X}}=\{X_{(k,\ell)}: 1 \le k \ne \ell \le n\}$. (Recall that we viewed these weights as exponential {\em rates}, then used the resulting exponential clocks at each step to determine which edge to add.) At each step, add an oriented edge whose tail and head are both the roots of some tree of the current forest, each such edge chosen with probability proportional to its weight. For this procedure, conditional on $\mathbf{X}$, after adding the first $i-1$ edges, the conditional probability of adding a particular edge $(k,\ell)$ is
\[
\frac{X_{(k,\ell)}}{\sum_{1 \le j \ne m \le n} X_{(r(T_j^{(i)}),r(T_m^{(i)}))}}\, .
\]
Now fix any sequence $f_1,\ldots,f_n$ of forests that can arise in the process, write $f_i = (t^{(i)}_{k},1 \le k \le n+1-i)$, 
and for $i=1,\ldots,n-1$ write $(k_i,\ell_i)$ for the unique edge of $f_{i+1}$ not in $f_{i}$. Then we have 
\begin{equation*}
\P{F_i=f_i,1 \le i \le n ~|~ \mathbf{X}} = \prod_{i=1}^{n-1} \frac{X_{(k_i,\ell_i)}}{\sum_{1 \le m \ne j \le n} X_{(r(t_m^{(i)}),r(t_j^{(i)}))}}\, .
\end{equation*}
It follows from the above analysis that for any such sequence $f_1,\ldots,f_n$, 
\[
\E{\prod_{i=1}^{n-1} \frac{X_{(k_i,\ell_i)}}{\sum_{1 \le m \ne j \le n} X_{(r(t_m^{(i)}),r(t_j^{(i)}))}}} = \frac{1}{n!(n-1)!}\, .
\]

Once again, it is not even {\em a priori} clear that this expectation should not depend on the law of $X$.
\bex[First-passage percolation]
Develop and analyze a ``Version 3'' variant of the tree growth procedure from Exercise~\ref{ex:random_recursive_trees}, using exponential edge weights. 
\eex
\begin{marginfig}%
\iftoggle{bookversion}{
  \includegraphics[width=1.25\linewidth,angle=270]{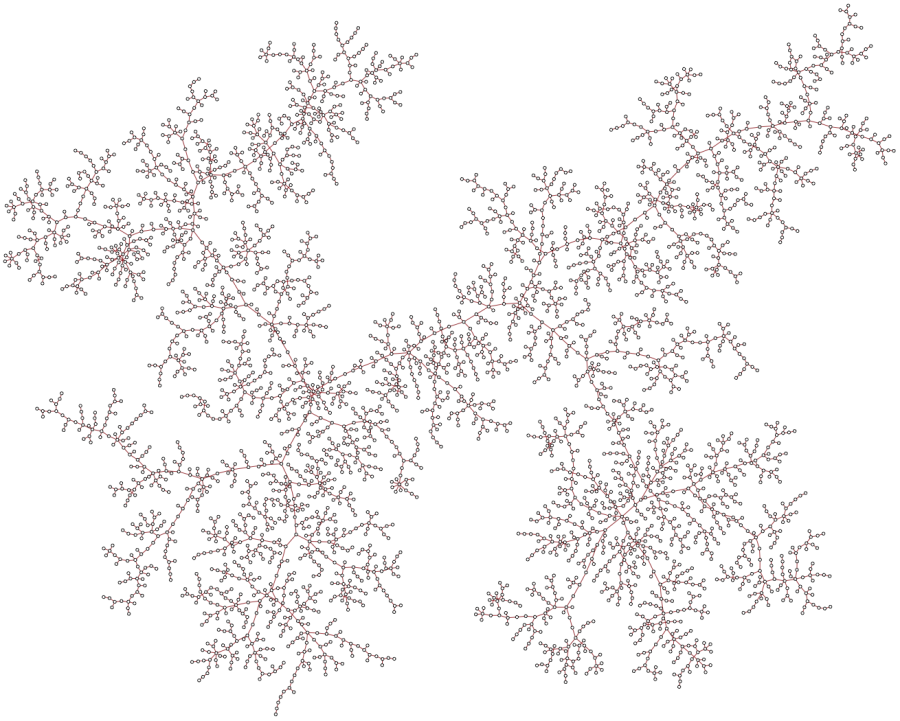}
  }
  {  \includegraphics[width=0.7\linewidth,angle=270]{mst}}
  \caption{The tree resulting from the multiplicative coalescent on $3000$ points.}
  \label{mfig:mult_coal_tree}
\end{marginfig}

\secsub{The multiplicative coalescent and minimum spanning trees}\label{intro:mult}
The previous two sections considered merging rules of the form any-to-root and root-to-root, and obtained Pitman's coalescent and Kingman's coalescent, respectively. We now take up the ``any-to-any'' merging rule. This is arguably the most basic of the three rules, but its behaviour is arguably the hardest to analyze. . We begin as usual from a forest $F_1$ of $n$ isolated vertices $\{1,\ldots,n\}$, and write $F_i=\{T_1^{(i)},\ldots,T_{n+1-i}^{(i)}\}$. 
In the multiplicative coalescent there is no natural way to maintain the property that edges are oriented toward some root vertex, so we view the trees of the forests as unrooted, and their edges as unoriented. 

\medskip
\isfullwidth{
\begin{mdframed}[style=algorithm]
{\bf The multiplicative coalescent.}
To obtain $F_{i+1}$ from $F_i$, choose an pair $\{U_i,V_i\}$ uniformly at random from the set of pairs $\{u,v\} \in {[n] \choose 2}$ for which $u$ and $v$ are different trees of $F_i$. Add an edge from $U_i$ to $V_i$ to form the forest $F_{i+1}$. 
\end{mdframed}
}
\medskip
This is known as the {\em multiplicative} coalescent, because the number of possible choices of an edge joining trees $T_{j}^{(i)}$ and $T_{k}^{(i)}$ is $|T_{j}^{(i)}||T_{k}^{(i)}|$. 
It follows that the number of possible edges that may be added to the forest $F_i$ is 
\[
\sum_{1 \le j \ne k \le n+1-i} |T_{j}^{(i)}||T_{k}^{(i)}| = \frac{1}{2} \left(n^2 - \sum_{T \in F_i} |T|^2\right). 
\]
The above expression is more complicated than for the additive coalescent or Kingman's coalescent: it depends on the forest $F_i$, for one. 

In much of the remainder of these notes, we investigate an expression for the partition function $Z_{\textsc{mc}}(n)$ of the multiplicative coalescent that arises from the preceding formula. 
To obtain this expression, recall the definition of an $n$-chain from Section~\ref{sec:intro}, and that $\mathcal{P}_n$ is the set of $n$-chains.
\bex
Show that $|\mathcal{P}_n| = \frac{(n!)^2}{n\cdot 2^{n-1}}$. 
\eex

The multiplicative coalescent determines an $n$-chain in which the $i$'th partition is simply $P(F_i):=\{V(T_{j}^{(i)}),1 \le j \le n+1-i\}$. It is straightforward to see that the number of possibilities for the multiplicative coalescent that give rise to a particular $n$-chain $P=(P_1,\ldots,P_n)$ is simply 
\[
\prod_{i=1}^{n-1} |A_i(P)||B_i(P)|\, ,
\]
where $A_i(P)$ and $B_i(P)$ are the parts of $P_i$ that are combined in $P_{i+1}$. 
It follows that 
\[
Z_{\textsc{mc}}(n) = \sum_{P=(P_1,\ldots,P_n) \in \mathcal{P}_n} \prod_{i=1}^{n-1} \pran{|A_i(P)||B_i(P)|}\, .
\]
This certainly looks more complicated than in the previous two cases. However, there is an exact formula for $Z_{\textsc{mc}}(n)$ whose derivation is perhaps easier than for either $Z_{\textsc{ac}}(n)$ or $Z_{\textsc{kc}}(n)$ (though it does rely on Cayley's formula). 
\begin{prop}\label{prop:zmc_exact}
$Z_{\textsc{mc}}(n) = n^{n-2}(n-1)!$
\end{prop}
\begin{proof}
Let $\mathcal{S}$ be the set of pairs $(t,\ell)$ where $t$ is an unrooted tree with $V(t)=[n]$ and $\ell:E(t) \to [n-1]$ is a bijection. By Cayley's formula, the number of trees $t$ with $V(t)=[n]$ is $n^{n-2}$, so $\mathcal{S}=n^{n-2}(n-1)!$. 

For $e \in E(T^{(1)}_n)$, let $L(e) = \sup\{i: e \not\in E(F_i)\}$. Then $L:E(T^{(1)}_n) \to [n-1]$ is a bijection. Thus the pair $(T^{(1)}_n,L)$ is an element of $\mathcal{S}$. To see this map is bijective, note that if $(T^{(1)}_n,L)=(t,\ell)$ then for each $1 \le i \le n$, $F_i$ is the forest on $[n]$ with edges $\{\ell^{-1}(j),1 \le j < i\}$. The result follows. 
\end{proof}
The above proposition yields that $Z_{\textsc{mc}}(n)=Z_{\textsc{ac}}(n)/n$. If we were to additionally choose a root for $T^{(1)}_n$, we would obtain identical partition functions. This suggests that perhaps the additive and multiplicative coalescents have similar structures. One might even be tempted to believe that the trees built by the two coalescents are identically distributed; the following exercise (an observation of Aldous \cite{aldous90random}), will disabuse you of that notation. 
\bex
Let $T$ be built by the multiplicative coalescent, and let $T'$ be obtained from the additive coalescent by unrooting the final tree. Show that if $n \ge 4$ then $T$ and $T'$ are not identically distributed. 
\eex
Despite the preceding exercise, it is tempting to guess that the two trees are still similar in structure; this was conjectured by Aldous \cite{aldous90random}, and only recently disproved \citep{addario13limit}. In the remainder of the section, we begin to argue for the difference between the two coalescents, from the perspective of their partition functions. For $1 \le k \le n$, write $Z_{\textsc{mc}}(n,k)$ for the partition function of the first $k$ steps of the multiplicative coalescent, 
\[
Z_{\textsc{mc}}(n,k) = \sum_{P=(P_1,\ldots,P_k) \in \mathcal{P}_{n,k}} \prod_{i=1}^{k-1} \pran{|A_i(P)||B_i(P)|}, 
\]
where $\mathcal{P}_{n,k}$ is the set of length-$k$ initial segments of $n$-chains. We have, e.g., $Z_{\textsc{mc}}(n,1)=1$, $Z_{\textsc{mc}}(n,2)={n \choose 2}$, and $Z_{\textsc{mc}}(n,n)=Z_{\textsc{mc}}(n)$. 

The argument of Proposition~\ref{prop:zmc_exact} shows that $Z_{\textsc{mc}}(n,k) = u_{n,k}\cdot (k-1)!$, where $u_{n,k}$ is the number of unrooted forests with vertices $[n]$ and $k-1$ total edges. The identity 
\[
u_{n,k} = {n \choose n+1-k} n^{k-2} \sum_{i=0}^{n+1-k}\left(\frac{-1}{2n}\right)^i {n+1-k \choose i}(n+1-k+i) \cdot (k-1)_i, 
\]
was derived by R\'enyi \cite{renyi59remarks}, and I do not know of an exact formula that simplifies the above expression. We begin to see that there is more to the multiplicative coalescent than first meets the eye. 

If we can't have a nice, simple identity, what about bounds? Of course, there is the trivial upper bound $Z_{\textsc{mc}}(n,k) \le (n(n-1)/2)^{k-1}$, since at each step there are at most ${n \choose 2}$ pairs to choose from; similar bounds hold for the other two coalescents. 
To improve this bound, and more generally to develop a deeper understanding of the dynamics of the multiplicative coalescent, our starting point is the following observation. 

Given an $n$-chain $P=(P_1,\ldots,P_n)$, for the multiplicative coalescent we have 
\[
\P{(P(F_i),1 \le i \le n)=P} = \prod_{i=1}^{n-1} \frac{2|A_i(P)||B_i(P)|}{n^2-\sum_{\pi \in P_i} |\pi|^2}\, .
\]
This holds since for $1 \le i \le n-1$, given that $P(F_j)=P_j$ for $1 \le j \le i$, there are $(n^2-\sum_{\pi \in P_i} |\pi|^2)/2$ choices for which oriented edge to add to form $F_{i+1}$, and $P(F_{i+1})=P_{i+1}$ for precisely $|A_i(P)||B_i(P)|$ of these. It follows that 
\begin{align}
Z_{\textsc{mc}}(n) & = \sum_{P=(P_1,\ldots,P_n) \in \mathcal{P}_n} \P{(P(F_i),1 \le i \le n)=P}\cdot \prod_{i=1}^{n-1} \frac{n^2-\sum_{\pi \in P_i} |\pi|^2}{2}\nonumber\\
	& = \sum_{P=(P_1,\ldots,P_n) \in \mathcal{P}_n} \P{(P(F_i),1 \le i \le n)=P}\cdot 2^{-(n-1)} \nonumber\\
	& \qquad\qquad\qquad \cdot 
	\Cexp{\prod_{i=1}^{n-1} \pran{n^2-\sum_{T \in F_i} |T|^2}}{(P(F_i),1 \le i \le n)=P} \nonumber\\
	& = 2^{-(n-1)}\cdot \E{\prod_{i=1}^{n-1}\pran{n^2-\sum_{T \in F_i} |T|^2}}. \label{eq:zmc_partition_expectation}
\end{align}
A mechanical modification of the logic leading to (\ref{eq:zmc_partition_expectation}) yields the following expression, valid for each $1 \le k \le n$: 
\begin{equation}\label{eq:zmc_partition_truncated_expectation}
Z_{\textsc{mc}}(n,k) = 2^{-(k-1)}\E{\prod_{i=1}^{k-1}\pran{n^2-\sum_{T \in F_i} |T|^2}}.
\end{equation}

Write 
\[
\hat{Z}_{\textsc{mc}}^{\to}(n,k) = \prod_{i=1}^{k-1} \pran{n^2-\sum_{T \in F_i} |T|^2}\, ,
\]
let $\hat{Z}_{\textsc{mc}}^{\to}(n) = \hat{Z}_{\textsc{mc}}^{\to}(n,1)$, and let $\hat{Z}_{\textsc{mc}}(n,k)=2^{-(k-1)}\hat{Z}_{\textsc{mc}}^{\to}(n,k)$ and $\hat{Z}_{\textsc{mc}}(n)=\hat{Z}_{\textsc{mc}}(n,n)$. With this notation, (\ref{eq:zmc_partition_expectation}) and the subsequent equation state that 
\begin{equation}\label{eq:zmc_partition_expectation_rewrite}
\E{\hat{Z}_{\textsc{mc}}(n,k)} = Z_{\textsc{mc}}(n,k)= \frac{1}{2^{k-1}}\E{\hat{Z}_{\textsc{mc}}^{\to}(n,k)}\, .
\end{equation}
The random variable $\hat{Z}_{\textsc{mc}}(n)$ is a sort of {\em empirical partition function} of the multiplicative coalescent. The superscript arrow on $\hat{Z}^{\to}_{\textsc{mc}}(n,k)$ is because the factor $2^{k-1}$ may be viewed as corresponding to a choice of orientation for each edge of $F_k$. The random variable $\hat{Z}_{\textsc{mc}}(n)$ of course contains more information than simply its expected value, so by studying it we might hope to gain a greater insight into the behaviour of the coalescent. Much of the remainder of these notes is devoted to showing that $\E{\hat{Z}_{\textsc{mc}}(n)}=Z_{\textsc{mc}}(n) $ is a {\em terrible} predictor of the typical value of $\hat{Z}_{\textsc{mc}}(n)$. More precisely, there are unlikely execution paths along which the multiplicative coalescent has many more possibilities than along a typical path; such paths swell the {\em expected} value of $\hat{Z}_{\textsc{mc}}(n)$ to exponentially larger than its {\em typical} size. 

The logic leading to (\ref{eq:zmc_partition_expectation}) and (\ref{eq:zmc_partition_truncated_expectation}) may also be applied to the additive coalescent; the result is boring but instructive. First note that 
\[
Z_{\textsc{mc}}(n,k) = \sum_{P=(P_1,\ldots,P_k) \in \mathcal{P}_{n,k}} \prod_{i=1}^{k-1} \pran{|A_i(P)|+|B_i(P)|}. 
\]
For the additive coalescent, the total number of choices at step $i$ is $n(n-i)$, and given that $P(F_i)=P_i$, the number of choices which yield $P(F_{i+1})=P_{i+1}$ is $A_i(P)+B_i(P)$. writing $\mathbf{P}_{\textsc{ac}}$ for probabilities under the additive coalescent, we thus have 
\[
\mathbf{P}_{\textsc{ac}}\left\{(P(F_i),1 \le i \le k)-(P_1,\ldots,P_k)\right\} = \prod_{i=1}^{k-1} \frac{|A_i(P)|+|B_i(P)|}{n(n-i)}
\]
Following the logic through yields 
\[
Z_{\textsc{ac}}(n,k) = \mathbf{E}_{\textsc{ac}}\left[\prod_{i=1}^{k-1} n(n-i)\right] = \mathbf{E}_{\textsc{ac}}\left[n^{k-1}(n-1)_{k-1}\right]. 
\]
Thus, the ``empirical partition function'' of the additive coalescent is a constant, so contains no information beyond its expected value. (This fact is essentially the key to Pitman's proof of Cayley's formula.) 

The terms of the products (\ref{eq:zmc_partition_expectation}) and (\ref{eq:zmc_partition_truncated_expectation}), though random, turn out to behave in a very regular manner (but proving this will take some work). Through a study of these terms, we will obtain control of $\E{\log\hat{Z}_{\textsc{mc}}(n)}$, and thereby justify the above assertion that $\hat{Z}_{\textsc{mc}}(n)$ is typically very different from its mean. 
\subpar{The growth rate of $Z_{\textsc{mc}}(n,\lfloor n/2\rfloor)$} \label{sec:ercoal}
As a warmup, and to introduce a key tool, we approximate the value of $Z_{\textsc{mc}}(n,\lfloor n/2\rfloor)$ 
using a connection between the multiplicative coalescent and a process we call (once again with a very slight abuse of terminology) the Erd\H{o}s-R\'enyi coalescent. Write $K_n$ for the {\em complete graph}, i.e. the graph with vertices $[n]$ and edges $(\{i,j\}, 1 \le i < j \le n)$. 

\medskip
\isfullwidth{
\begin{mdframed}[style=algorithm]
{\bf The Erd\H{o}s-R\'enyi coalescent.}
Choose a uniformly random permutation $e_1,\ldots,e_{n \choose 2}$ of $E(K_n)$. For $0 \le i \le {n \choose 2}$, let $G^{(n)}_i$ have vertices $[n]$ and edges $\{e_1,\ldots,e_i\}$. 
\end{mdframed}
}
\medskip

Our indexing here starts at zero, unlike in the multiplicative coalescent; this is slightly unfortunate, but it is standard for the Erd\H{o}s-R\'enyi graph process to index so that $G^{(n)}_i$ has $i$ edges. 
This process is different from the previous coalescent processes, most notably because it creates graphs with cycles. 

Note that we can recover the multiplicative coalescent from the Erd\H{o}s-R\'enyi coalescent in the following way. Informally, simply ignore any edges added by the Erd\H{o}s-R\'enyi coalescent that fail to join distinct components. More precisely, for each $0 \le m \le {n \choose 2}$, let $\tau_m$ be the number of edges 
$\{U_i,V_i\}$, $0 < i \le m$ such that $U_i$ and $V_i$ lie in different components of $G^{(n)}_{i-1}$. (See Figure~\ref{fig:mc_coal} for an example.) 
\begin{figure}[ht]
\begin{subfigure}[b]{0.17\textwidth}
		\includegraphics[width=\textwidth,page=1]{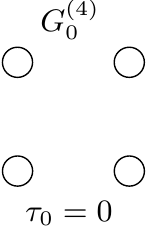}
\end{subfigure}%
\quad
\begin{subfigure}[b]{0.17\textwidth}
		\includegraphics[width=\textwidth,page=2]{graphics/gnpex.pdf}
\end{subfigure}%
\quad
\begin{subfigure}[b]{0.17\textwidth}
		\includegraphics[width=\textwidth,page=3]{graphics/gnpex.pdf}
\end{subfigure}%
\quad
\begin{subfigure}[b]{0.17\textwidth}
		\includegraphics[width=\textwidth,page=4]{graphics/gnpex.pdf}
\end{subfigure}%
\quad
\begin{subfigure}[b]{0.17\textwidth}
		\includegraphics[width=\textwidth,page=5]{graphics/gnpex.pdf}
\end{subfigure}%
\caption{An example of the first steps of the Erd\H{o}s-R\'enyi coalescent. The multiplicative coalescent is obtained by keeping only the thicker, blue edges.} 
\label{fig:mc_coal}
\end{figure}
Observe that 
\[
\tau_m+1=	\begin{cases}
			\tau_m	&\mbox{ if }G^{(n)}_{m+1}\mbox{ and }G^{(n)}_m\mbox{ have the same number of components} \\
			\tau_m+1	&\mbox{ if } G^{(n)}_{m+1}\mbox{ has one fewer component than }G^{(n)}_m\, .
		\end{cases}
\]
In other words, $\tau_m$ increases precisely when the the endpoints of the edge added to $G^{(n)}_m$ are in different components. Further, the set 
\[
\left\{e_{m}: m \ge 1,\tau_m>\tau_{m-1}\right\}
\] 
contains $n-1$ edges, since $G^{(n)}_0$ has $n$ components and
 $G^{(n)}_{{n \choose 2}}$ almost surely has only one component. 

Set $I_1=0$ and for $1 < k \le n$ let 
\[
I_k = \inf\{m \ge 1: \tau_m=k-1\}\, .
\] 
Then for $1 < k \le n$, the edge $e_{I_k}$ joins distinct components of $G^{(n)}_{I_k-1}$, and by symmetry is equally likely to be any such edge. Thus, letting $F_k$ be the graph with edges $\{e_{I_j}: 1\le j\le k\}$ for $1 \le k \le n$, the process $\{F_k,1 \le k \le n\}$ is precisely distributed as the multiplicative coalescent. This is a {\em coupling} between the Erd\H{o}s-R\'enyi graph process and the multiplicative coalescent; its key property is that for all $1 \le k \le n$, the vertex sets of the trees of $F_k$ are the same as those of the components of $G^{(n)}_{I_k}$. 

Having found the multiplicative coalescent within the Erd\H{o}s-R\'enyi coalescent, we can now use known results about the latter process to study the former. For a graph $G$, and $v \in V(G)$, we write $N(v)=N_G(v)$ for the set of nodes adjacent to $v$, and write $C(v)=C_G(v)$ for the connected component of $G$ containing $v$. 
We will use the results of the following exercise.\footnote{Until further notice, we omit ceilings and floors for readability.} 
\bsex\label{ex:comp_sizes_errg_weak}
\begin{itemize}
\item[(a)] Show that in the Erd\H{o}s-R\'enyi coalescent, if all components have size at most $s$ then the probability a uniformly random edge from among the remaining edges has both endpoints in the same component is at most $(s-1)/(n-3)$.
\item[(b)] Show that for all $0 \le m \le n/2$, in $G^{(n)}_k$, $\mathbb{E}|N(v)| \le 2m/n$. 
\item[(c)] Prove by induction that for all $0 \le m < n/2$, in $G^{(n)}_m$, $\E{|C(1)|} \le n/(n-2m)$. \\({\bf Hint.} First condition on $N(1)$, then average.)
\item[(d)] Prove that for all $\eps > 0$, 
\[\limsup_{n \to \infty} \mathbb{P}\left(G^{(n)}_{(1-\eps)n/2}\mbox{ has a component of size }>\eps n\right) \to 0\, .\]
({\bf Hint.} Given that the largest component of $G^{(n)}_m$ has size $s$, with probability at least $s/n$ vertex $1$ is in such a component.) 
\end{itemize}
\esex
Using the above exercise, we now fairly easily prove a lower bound on the partition function of the first half of the multiplicative coalescent. 
\begin{prop}\label{prop:mult_coal_lower_bound_weak}
For all $\beta > 0$, 
\[
\p{\hat{Z}_{\textsc{mc}}^{\to}(n,\lfloor n/2\rfloor) \ge n^{(1-\beta)n}} \to 1\, \mbox{ as } n \to \infty\, .
\]
\end{prop}
We begin by showing that typically $I_t = (1+o(1))t$ until $t \ge n/2$. 
\begin{lem}\label{lem:firsthalf_nocycles}
For all $\eps > 0$, $\limsup_{n \to \infty} \mathbb{P}\left(I_{(1-\eps)n/2} \ge n/2\right) =0$.
\end{lem}
\begin{proof}
Fix $\eps > 0$, let $\delta=\eps/3$, and let $E$ be the event that all components of $G_{n(1-\delta)n/2}$ have size at most $\delta n$. 
For $m \ge 0$, conditional on $G^{(n)}_m$, by Exercise~\ref{ex:comp_sizes_errg_weak}~(a), $\tau_{m+1}-\tau_m$ stochastically dominates a Bernoulli$(1-(s-1)/(n-3))$ random variable, where $s$ is the largest component of $G^{(n)}_m$. 

For $n$ large and $s \le \delta n$ we have $1-(s-1)/(n-3) \ge 1-\eps/2$. Therefore, on $E$ and for large $n$ the sequence $(\tau_{m+1}-\tau_m,0 \le m < (1-\delta)n/2)$ stochastically dominates a sequence $(B_m,0 \le m < (1-\delta)n/2)$ of iid Bernoulli$(1-\eps/2)$ random variables. It follows that 
\begin{align*}
\p{\tau_{(1-\delta)n/2 \le (1-\eps)n/2}} 	& \le \p{E_n} + \p{\tau_{(1-\delta)n/2 \le (1-\eps)n/2}~|~E_n^c} \\
									& \le \p{E_n} + \p{\mathrm{Bin}((1-\delta)n/2,1-\eps/2) < (1-\eps)n/2} \\
									& = o(1)\, ,
\end{align*}
the last line Exercise~\ref{ex:comp_sizes_errg_weak}~(d) and Chebyshev's inequality (note that $(1-\delta)(1-\eps/2)n/2 > (1-5\eps/6)n/2$). 
On the other hand, if $\tau_{(1-\delta)n/2} > (1-\eps)n/2$ then $I_{(1-\eps)n/2} \le (1-\delta)n/2 < n/2$. 
\end{proof}
\begin{proof}[Proof of Proposition~\ref{prop:mult_coal_lower_bound_weak}]
View $(F_1,\ldots,F_n)$ as coupled with the by the Erd\H{o}s-R\'enyi coalescent as above, so that $F_k$ and $G_{I_k}^{(n)}$ have the same components. 
Fix $\delta \in (0,1/4)$ and let $k=k(n)= n/2-2\delta n$. Let $E_1$ be the event that $I_{n/2-\delta n} < n/2$.\footnote{We omit the dependence on $n$ in the notation for $E_1$; similar infractions occur later in the proof.} Since $I_{m+1} \ge I_m + 1$ for all $m$, we have 
\[
I_k \le I_{n/2-\delta n} - \left((n/2-\delta n)-k\right) = I_{n/2-\delta n}-\delta n\, .
\] 
Thus, on $E_1$ we have $I_k \le (1-2\delta)n/2$. 

Next let $E_2$ be the event that all component sizes in $G^{(n)}_{(1-2\delta)n/2}$ are at most $\delta n$. 
The components of $F_k$ are precisely the components of $G^{(n)}_{I_k}$, so if $E_1\cap E_2$ occurs then since on $E_1$ we have $I_k \leq (1-2\delta)n/2$, all components of $F_k$ have size at most $\delta n$. In this case, for all $i \le k$ the components of $F_i$ clearly also have size at most $\delta n$. 

It follows\footnote{To maximize $\sum_{j} x_j^2$ subject to the conditions that $\sum_j x_j=1$ and that $\max_j x_j \le \delta$, take $x_j=\delta$ for $1\le j \le \delta^{-1}$.} that on $E_1 \cap E_2$, for all $i \le k$, 
\[
\sum_{T \in F_i} |T|^2 \le \delta n^2 \, 
\]
so on $E_1 \cap E_2$, 
\begin{align}
\hat{Z}^{\to}_{\textsc{mc}}(n,k+1) & = \prod_{i=1}^k \left(n^2 - \sum_{T \in F_i} |T|^2\right) \nonumber\\
						& \ge n^{2k}(1-\delta)^k  \label{eq:mult_coal_lower_bound_weak}\\
						& = n^{n-(4\delta -\log(1-\delta)) n} \nonumber
\end{align}
By Exercise~\ref{ex:comp_sizes_errg_weak}~(d) and Lemma~\ref{lem:firsthalf_nocycles}, $\mathbb{P}\left(E_1 \cap E_2\right) \to 1$ as $n \to \infty$. 
Since $\hat{Z}^{\to}_{\textsc{mc}}(n,\lfloor n/2\rfloor) \ge \hat{Z}^{\to}_{\textsc{mc}}(n,k+1)$ for $n$ large, the result follows. 
\end{proof}
The following exercise is to test whether you are awake. 
\bex \label{ex:zn_half_log_asymptotic}
Prove that 
\[
\frac{\log Z_{\textsc{mc}}(n,\lfloor n/2\rfloor)}{n\log n} \to 1\, ,
\]
as $n \to \infty$. 
\eex
We next use Proposition~\ref{prop:mult_coal_lower_bound_weak} (more precisely, the inequality (\ref{eq:mult_coal_lower_bound_weak}) obtained in the course of its proof) to obtain a first lower bound on $Z_{\textsc{mc}}(n)$. 
\begin{cor}\label{cor:mult_coal_lower_bound_weak}
It holds that 
\[
\frac{Z_{\textsc{mc}}(n,\lfloor n/2\rfloor)}{Z_{\textsc{ac}}(n,\lfloor n/2\rfloor)} = \pran{\frac{e}{4}}^{(1+o(1))n/2}\, .
\]
\end{cor}
\begin{proof}
By Proposition~\ref{prop:mult_coal_lower_bound_weak} and (\ref{eq:zmc_partition_expectation_rewrite}), we have 
\[
Z_{\textsc{mc}}(n,\lfloor n/2\rfloor) \ge 2^{-(\lfloor n/2\rfloor-1)} n^{(1+o(1)) n}, 
\]
so by Exercise~\ref{ex:ac_partial_partitition}, 
\[
\frac{Z_{\textsc{mc}}(n,\lfloor n/2\rfloor)}{Z_{\textsc{ac}}(n,\lfloor n/2\rfloor)} = \frac{n^{(1+o(1)) n}}{2^{\lfloor n/2\rfloor-1}n^{\lfloor n/2\rfloor-1} (n-1)_{\lfloor n/2\rfloor-1}} = \frac{n^{(1+o(1))n}(n/2)!}{2^{n/2}n^{n/2} n!}\, .
\]
Using Stirling's approximation\footnote{Stirling's approximation says that $m!/(\sqrt{2\pi m} (m/e)^m) \to 1$ as $m \to \infty$; in fact the (much less precise) fact that $\log(m!) = m\log m - m+o(m)$ is enough for the current situation.}, it follows easily that 
\[
\frac{Z_{\textsc{mc}}(n,\lfloor n/2\rfloor)}{Z_{\textsc{ac}}(n,\lfloor n/2\rfloor)} \ge \pran{\frac{e}{4}}^{(1+o(1))n/2}. 
\]
The corresponding upper bound follows similarly, using that $Z_{\textsc{mc}}(n,\lfloor n/2\rfloor) \le (n(n-1)/2)^{\lfloor n/2\rfloor-1}=n^{n(1+o(1))}/2^{n/2}$. 
\end{proof}
\bex
Perform the omitted calculation using Stirling's formula from the proof of Corollary~\ref{cor:mult_coal_lower_bound_weak}. 
\eex
The preceding corollary is evidence that despite the similarity of the partition functions $Z_{\textsc{mc}}(n)$ and $Z_{\textsc{ac}}(n)$, the fine structure of the multiplicative coalescent is may be interestingly different from that of the additive coalescent. 

\subpar{The multiplicative coalescent and Kruskal's algorithm} \label{sec:mult_coal_other_versions}
There is a pleasing interpretation of ``Version 2'' of the multiplicative coalescent, which is driven by exchangeable distinct edge weights ${\mathbf{W}}=\{W_{\{j,k\}}, 1 \le  j<k \le n\} = \{W_e,e \in E(K_n)\}$. (A special case is that the elements of $\mathbf{W}$ are iid continuous random variables.). The symmetry of the model makes it straightforward to verify that this results in a sequence $(F_1,\ldots,F_n)$ with the same distribution as the multiplicative coalescent.

\medskip
\isfullwidth{
\begin{mdframed}[style=algorithm]
\noindent {\bf Multiplicative Coalescent Version 2: Kruskal's algorithm.}\\ 
Let $F_1$ be a forest of $n$ isolated vertices $1,\ldots,n$. \\
For $1 \le i < n$:\\
\hspace*{0.2cm}$\star$ Let $\{j,k\}\in E(K_n)$ minimize 
\iftoggle{bookversion}{}{}
$\{W_{\{j,k\}}: j,k \mbox{ in distinct trees of }F_i\}$. \\
\hspace*{0.2cm}$\star$ Form $F_{i+1}$ from $F_i$ by adding $\{j,k\}$. 
\end{mdframed}
}
\medskip
\bsex
Prove that any exchangeable, distinct edge weights $\mathbf{W}=\{W_e,e \in E(K_n)\}$ again yield a process with the law of the multiplicative coalescent. 
\esex
At step $i$, the edge-weight driven multiplicative coalescent simply adds the smallest weight edge whose endpoints lie in distinct components of $F_i$. 
In other words, it adds the smallest weight edge whose addition will not create a {\em cycle} in the growing graph. This is simply {\em Kruskal's algorithm} for building the minimum weight spanning tree. 
When the weights $W_{\{j,k\}}$ are all non-negative, the tree obtained at the end of the Version 2 multiplicative coalescent, $T_1^{(n)}$, is the minimum weight spanning tree of $K_n$ with weights $\mathbf{W}$. We denote it 
$\mathrm{MST}(K_n,\mathbf{W})$, and refer to it as the {\em random MST of $K_n$}.

Order $E(K_n)$ by increasing order of $\mathbf{W}$-weight as $e_1,\ldots,e_{n \choose 2}$. The exchangeability of $\mathbf{W}$ implies this is a uniformly random permutation of $E(K_n)$. Letting $G^{(n)}_k$ have edges $e_1,\ldots,e_k$ thus yields an important instantiation of our coupling of the Erd\H{o}s-R\'enyi coalescent and the multiplicative coalescent; we return to this in Section~\ref{sec:frieze_lim}.

\subpar{Other features of the multiplicative coalescent} \label{sec:mult_coal_remainder}
\iftoggle{bookversion}{%
}{%
The remainder of the section is not essential to the main development. The following exercise was inspired by a discussion with Remco van der Hofstad. 
}%
\bex[First-passage percolation]
Consider the multiplicative coalescent driven by exchangeable, distinct edge weights $\mathbf{W}$ and for $1 \le i < j \le n$, let $d(i,j) = \min \sum_{e \in \gamma} W_e$, the minimum taken over paths from $i$ to $j$ in $K_n$. Show that the minimum is attained by a unique path $\gamma_{i,j}$. Find exchangeable edge weights $\{W_e,e \in E(K_n)\}$ for which, for each for each $1 \le i < j \le n$, $\gamma_{i,j}$ is a path of $T_1^{(n)}$. 

\eex

Finally, we turn to Version 3 of the process, in which we view arbitrary iid non-negative weights ${\mathbf{X}}=\{X_{i,j}, 1 \le i < j \le n\}$ as {\em rates} for edge addition. In view of the preceding paragraph, this gives a process that results in a tree with the same {\em distribution} as the random MST of $K_n$, but which is not necessarily equal to the MST. In particular, the tree is {\em not} a deterministic function of the edge weights; for example, we may take $\mathbf{X}$ to be a deterministic vector such as the all-ones vector, whereas the resulting tree always is random.
\bex
Find (iid random) {\em rates} $\mathbf{X}$ for which, in version 3 of the process, the resulting tree $T_1^{(n)}$ is equal to the random MST of $K_n$ with {\em weights} $\mathbf{X}$, with probability tending to one as $n \to \infty$. \eex

\iftoggle{bookversion}{
In the next section, we consider differences between the structures of the trees formed by the three coalescents.

\secsub{The heights of the three coalescent trees}}
{
\section{Intermezzo: The heights of the three coalescent trees}
To date we have been primarily studying the partition functions of the coalescent processes. The processes have many other interesting features, however. In this section we discuss differences between the structures of the trees formed by the three coalescents.}

Write $T_{\textsc{kc}}^{(n)},T_{\textsc{ac}}^{(n)}$, and $T_{\textsc{mc}}^{(n)}$, respectively, for the trees formed by Kingman's coalescent, the additive coalescent, and the multiplicative coalescent. In each case the coalescent starts from $n$ isolated vertices $\{1,\ldots,n\}$, so each of these trees has vertices $\{1,\ldots,n\}$. If $T$ is any of these trees and $e$ is an edge of $T$, we write 
$L(e)=i$ if $e$ was the $i$'th edge added during the execution of the coalescent. Above, we established the following facts about the distributions of these random trees. 

\medskip
\begin{enumerate}
\item Ignoring vertex labels, $(T_{\textsc{kc}}^{(n)},L)$ is uniformly distributed over pairs $(t,\ell)$, where $t$ is a rooted tree with $n$ vertices and $\ell$ is a decreasing edge labelling of $t$. (We simply refer to such pairs as {\em decreasing trees with $n$ vertices}, for short.)
\item $T_{\textsc{ac}}^{(n)}$ is uniformly distributed over the set of rooted trees with vertices $\{1,\ldots,n\}$. (We refer to such trees as {\em rooted labeled trees with $n$ vertices}.)
\item $T_{\textsc{mc}}^{(n)}$ is distributed as the minimum weight spanning tree of the complete graph $K_n$, with iid continuous edge weights ${\mathbf{W}}=\{W_{i,j}, 1 \le i < j \le n\}$. 
\end{enumerate}
\medskip

What is known about these three distributions? To illustrate the difference between them, we consider a fundamental tree parameter, the {\em height}: this is simply the greatest number of edges in any path starting from the root.\footnote{A glance back at Figures~\ref{mfig:additive_coal_tree},~\ref{mfig:kingman_coal_tree} and~\ref{mfig:mult_coal_tree} gives a hint as to the relative heights of the three trees.} The third tree, $T_{\textsc{mc}}^{(n)}$ is not naturally rooted, but one may check that any choice of root will yield the same height up to a multiplicative factor of two; we root $T_{\textsc{mc}}^{(n)}$ at vertex $1$ by convention. Given a rooted tree $t$, we write $r(t)$ for its root and $h(t)$ for its height. The following exercise develops a fairly straightforward route to upper bounds on $h(T_{\textsc{kc}}^{(n)})$ that are tight, at least to first order. 

\bex\label{ex:poisson_tail}
Let $D_i$ be the number of edges on the path from vertex $i$ to $r(T_{\textsc{kc}}^{(n)})$. 
\begin{itemize}
\item[(a)] Show that $\{D_1,\ldots,D_n\}$ are exchangeable random variables. 
\item[(b)] Show that $D_1$ is stochastically dominated by a Poisson$(\log n)$ random variable. 
\item[(c)] Show that for $X$ a Poisson$(\mu)$ random variable, for $x \ge \mu$, 
$\P{X > x} \le e^{-\mu} (e\mu/x)^x$.
\item[(d)] Show that $\P{\max_{1 \le i \le n} D_i \ge e\log n} \to 0$ as $n \to \infty$.
\item[(e)] Show that $\limsup_{n \to \infty} (\max_{1 \le i \le n} D_i - e\log n) \to -\infty$ in probability. 
\end{itemize}
\eex

We next turn to $T_{\textsc{ac}}^{(n)}$. I am not aware of an easy way to directly use the additive coalescent to analyze the height of $T_{\textsc{ac}}^{(n)}$. However, one can use the additive coalescent to derive combinatorial results which, together with exchangeability, yields lower bounds of the correct order of magnitude, and upper bounds that are tight up to poly-logarithmic corrections; such bounds are the content of the following exercise. 
A non-negative random variable $R$ has the {\em standard Rayleigh} distribution if it has density $f(x)=xe^{-x^2/2}$ on $[0,\infty)$. 
\bex \label{ex:ac_height_bounds}
Let $D_i$ be the number of edges on the path from vertex $i$ to $r(T_{\textsc{ac}}^{(n)})$. 
\begin{itemize}
\item[(a)] Show that $\{D_1,\ldots,D_n\}$ are exchangeable random variables. 
\item[(b)] Show that the number of pairs $(t,i)$, where $t$ is a rooted labeled tree with $V(t)=[n]$ and $i \in V(t)$ has $d(r(t),i)=k-1$, is 
$k\cdot (n)_k \cdot n^{n-k-1}$
\item[(c)] 
Show that for $1 \le k \le n$, $\P{D_1=k-1} = \frac{k}{n} \prod_{i=1}^{k-1} \left(1-\frac{i}{n}\right)$. Conclude that $D_1/\sqrt{n}$ converges in distribution to a standard Rayleigh. 
\item[(d)] Using (c) and a union bound, show that if $(c_n,n \ge 1)$ are constants with $c_n \to \infty$ then $\P{\max_{1 \le i \le n} D_i > c_n \sqrt{n \log n}} \to 0$. 
\item[(e)] Use the exchangeability of the trees in a uniformly random ordered labeled forest to prove that $\P{|\{i: D_i \ge k/2 \}| \ge n/2~|~D_1 = k} \ge 1/2$ for all $1 \le k \le n$. 
\item[(f] Use (c) and (e) to show that f $(c_n,n \ge 1)$ are constants with $c_n \to \infty$ then $\P{\max_{1 \le i \le n} D_i > c_n \sqrt{n}} \to 0$, strengthening the result from (d). 
\end{itemize}
\eex
From the preceding exercise, we see immediately that $T_{\textsc{ac}}^{(n)}$ has a very different structure from $T_{\textsc{kc}}^{(n)}$, which had logarithmic height. Moreover, the heights of the two trees are {\em qualitatively} different. The height of $T_{\textsc{kc}}^{(n)}$ is concentrated: $h(T_{\textsc{kc}}^{(n)})/\log n \to e$ in probability. On the other hand, $h(T_{\textsc{ac}}^{(n)})$ is diffuse: $h(T_{\textsc{ac}}^{(n)})/n^{1/2}$ converges in distribution to a non-negative random variable with a density.\footnote{Neither of these convergence statements follows from the exercises, and both require some work to prove. The fact that $h(T_{\textsc{kc}}^{(n)})/\log n \to e$ in probability was first shown by \citet{devroye86union}. The distributional convergence of $h(T_{\textsc{ac}}^{(n)})/n^{1/2}$ is a result of \citet{renyi67height}.}
\footnote{In fact, if edge lengths in $T_{\textsc{ac}}^{(n)}$ are multiplied by $n^{-1/2}$ then the resulting object converges in distribution to a random compact metric space called the {\em Brownian continuum random tree} (or {\em CRT}), and $ h(T_{\textsc{ac}})/n^{1/2}$ converges in distribution to the height of the CRT. For more on this important result, we refer the reader to \cite{aldouscrtov91,legall05b}}

What about the tree $T_{\textsc{mc}}^{(n)}$ built by the multiplicative coalescent? Probabilistically, this is the most challenging of the three to study. 
For $T_{\textsc{kc}}^{(n)}$ and $T_{\textsc{ac}}^{(n)}$, Exercises~\ref{ex:poisson_tail} and~\ref{ex:ac_height_bounds} yielded exact or nearly exact expressions for the distance between the root and a fixed vertex (by exchangeability, this is equivalent to the distance between the root and a {\em uniformly random} vertex. The partition function $Z_{\textsc{mc}}(n)$ seems too complex for such a direct argument to be feasible. 

The coalescent procedure can be used to obtain lower bounds on the height, but with greater effort than in the two preceding cases. Our approach is elucidated by the following somewhat challenging exercise. Let $K_n$ have iid Exponential$[0,1]$ edge weights, and let $H$ be the subgraph of $K_n$ with the same vertices, but containing only edges of weight at most $1/n$. A {\em tree component} of $H$ is a connected component of $H$ that is a tree. 
\bex
\item[(a)] Let $N$ be the number of vertices in tree components of $H$, whose component has size $\lfloor n^{1/4} \rfloor$. Using Chebyshev's inequality, show that $\P{N=0} \to 0$ as $n \to \infty$.
\item[(b)] Fix $S \subset \{1,\ldots,n\}$. Show that, given that $H$ contains a tree component whose vertices are precisely $S$, then such a component is uniformly distributed over labeled trees with vertices $S$. 
\item[(c)] Use Kruskal's algorithm to show that any tree component of $H$ is a subtree of the minimum weight spanning tree of $T_{\textsc{mc}}^{(n)}$. 
\item[(d)] Use Exercise~\ref{ex:ac_height_bounds}~(c) to conclude that, as $n \to \infty$, 
\[
\P{T_{\textsc{mc}}^{(n)}~\mbox{ has height at most } \frac{n^{1/8}}{\log^2 n}} \to 0\, .
\]
\eex
This shows that $T_{\textsc{mc}}^{(n)}$ is quite different from $T_{\textsc{kc}}^{(n)}$.\footnote{With more care, one can show that with high probability $H$ contains tree components containing around $n^{2/3}$ vertices and with height around $n^{1/3}$, which yields that with high probability $T_{\textsc{mc}}^{(n)}$ has height of order at least $n^{1/3}$.} It is not as straightforward to bound the height of $T_{\textsc{mc}}^{(n)}$ away from $n^{1/2}$ using the tools currently at our disposal. It turns out that $T_{\textsc{mc}}^{(n)}$ has height of order $n^{1/3}$ (and has non-trivial fluctuations on this scale), but proving this takes a fair amount of work \cite{addario09critical} and is beyond the scope of these notes. 
\bex[Open problem -- two point function of random MSTs]
Let $D_n$ be the distance from vertex $n$ to vertex $1$ in $T_{\textsc{mc}}^{(n)}$. Obtain an explicit expression for the distributional limit of $D_n/n^{1/3}$. 
\eex

\chsec{The susceptibility process.}\label{ch:zmc_lower_bound}
The remainder of the paper focusses exclusively on the multiplicative coalescent, which we continue to denote $(F_1,\ldots,F_n)$. 
Recall that $\hat{Z}_{\textsc{mc}}^{\to}(n,k) = \prod_{i=1}^{k-1} \pran{n^2-\sum_{T \in F_i} |T|^2}$. 
The terms in the preceding product are not independent; linearity of expectation makes the ``empirical entropy'' $\log \hat{Z}_{\textsc{mc}}^{\to}(n)$ easier to study. 
\begin{equation}\label{eq:zmc_partition_jensenbound}
\E{\log\hat{Z}_{\textsc{mc}}^{\to}(n)} = \sum_{k=1}^{n-1} \E{\log(n^2-\sum_{T \in F_k} |T|^2)}\, .
\end{equation}
The expectation in the latter sum is closely related to the {\em susceptibility} of the forest $F_i$. More precisely, given a finite graph $G$, write $\mathcal{C}(G)$ for the set of connected components of $G$. The susceptibility of $G$ is the quantity 
\[
\chi(G) = \frac{1}{|G|} \sum_{C \in \mathcal{C}(G)} |C|^2. 
\]
Recalling that $C(v)=C_G(v)$ is the component of $G$ containing $v$, we may also write $\chi(G) = |G|^{-1} \sum_{v \in V(G)}|C(v)|$, so $\chi(G)$ is the expected size of the component containing a uniformly random vertex from $G$. 
\bex\label{eq:susc_comp_size_bounds}
Let $G$ be any graph, write $L$ and $S$ for the number of vertices in the largest and second-largest components of $G$, respectively. Then
\[
\frac{L^2}{|G|} \le \chi(G) \le \frac{L^2}{|G|} + S. 
\]
\eex
Viewing $F_i$ as a graph with vertices $\{1,\ldots,n\}$, (\ref{eq:zmc_partition_jensenbound}) becomes
\begin{equation}\label{eq:logpart_lowerbound_suscept}
\E{\log\hat{Z}_{\textsc{mc}}^{\to}(n)} = 2(n-1)\log n + \E{\sum_{k=1}^{n-1} \log\pran{1-\frac{\chi(F_k)}{n}}}\, .
\end{equation}
In order to analyze this expression, we use the connection with the Erd\H{o}s-R\'enyi coalescent $(G^{(n)}_m,0 \le m \le {n \choose 2})$, which we described in Section~\ref{sec:ercoal}; in brief, we coupled to $(F_k, 1 \le k \le n)$ by letting $F_k$ have edges $\{e_{I_j}, 1\le j \le k\}$, where $I_k$ was the first time $m$ that $G^{(n)}_m$ had $n+1-k$ components. 

\begin{prop}\label{prop:empirical_mc_identity}
\begin{align*}
\E{\log{\hat{Z}_{\textsc{mc}}^{\to}(n)}} & = 2(n-1)\log n\,  + \\
 & \sum_{m=0}^{{n \choose 2}- 1} \left(1-\frac{n+2m}{n^2}\right)^{-1}\E{\log \pran{1-\frac{\chi(G^{(n)}_{m})}{n}} \cdot \pran{1-\frac{\chi(G^{(n)}_{m})}{n}}}\, .
\end{align*}
\end{prop}
\begin{proof}
In the coupling with the Erd\H{o}s-R\'enyi coalescent, $F_k$ and $G^{(n)}_{I_k}$ have the same connected components, so $\chi(F_k)=\chi(G^{(n)}_{I_k}$. We obtain the identity 
\begin{align*}
\sum_{k=1}^{n-1} \log\pran{1-\frac{\chi(F_k)}{n}}
& = \sum_{k=1}^{n-1} \log \pran{1-\frac{\chi(G^{(n)}_{I_k})}{n}}\\
& = \sum_{m =0}^{{n \choose 2}-1} \log \pran{1-\frac{\chi(G^{(n)}_{m})}{n}} \I{\chi(G^{(n)}_{m+1}) > \chi(G^{(n)}_{m})}
\end{align*}
Using the tower law for conditional expectations, we thus have 
\begin{align}
		& \E{\sum_{k=1}^{n-1} \log\pran{1-\frac{\chi(F_k)}{n}}} \nonumber\\
		= & \sum_{m=0}^{{n \choose 2}-1} \E{ \log \pran{1-\frac{\chi(G^{(n)}_{m})}{n}} \I{\chi(G^{(n)}_{m+1}) > \chi(G^{(n)}_{m})}}. \nonumber\\
		= & \sum_{m=0}^{{n \choose 2}-1} \E{ \E{\log \pran{1-\frac{\chi(G^{(n)}_{m})}{n}} \I{\chi(G^{(n)}_{m+1}) > \chi(G^{(n)}_{m})}~|~G^{(n)}_{m}}} \nonumber\\
		= & \sum_{m=0}^{{n \choose 2}-1} \E{\log \pran{1-\frac{\chi(G^{(n)}_{m})}{n}} \cdot \p{\chi(G^{(n)}_{m+1}) > \chi(G^{(n)}_{m})~|~G^{(n)}_{m}}} \nonumber
\end{align}
For any finite graph $G$, the quantity $\chi(G)/|G| = \sum_{C \in \cC(G)}|C|^2/|G|^2$ is simply the probability that a pair $(U,V)$ of independent, uniformly random vertices of $G$ lie in the same component of $G$. Let $(U,V)$ be independent, uniformly random elements of $[n]=G^{(n)}_{m-1}$. Then 
\[
\p{U \ne V,\{U,V\} \not \in E(G^{(n)}_{m}) ~|~ G^{(n)}_{m}}  = 1-\frac{1}{n} - \frac{2|E(G^{(n)}_{m})|}{n^2} 
											 = 1- \frac{n+2m}{n^2} 
\]
Conditionally given that $U \ne V$ and $\{U,V\} \not \in E(G^{(n)}_{m})$, the pair $\{U,V\}$ has the same law as $e_{m+1}$. 
It follows that 
\begin{align}
	& \p{\chi(G^{(n)}_{m+1}) > \chi(G^{(n)}_m)~|~G^{(n)}_m} \nonumber\\
 = 	& \p{C_{G^{(n)}_m}(U) \ne C_{G^{(n)}_m}(V)~|~G^{(n)}_m,U \ne V,\{U,V\} \not \in E(G^{(n)}_m)} \nonumber\\
 = 	& \left(1-\frac{\chi(G^{(n)}_{m-1})}{n}\right)\left(1- \frac{n+2m}{n^2}\right)^{-1}\, , \label{eq:prob_to_exp_susc}
\end{align}
so 
\begin{align}
		& \E{\sum_{k=1}^{n-1} \log\pran{1-\frac{\chi(F_k)}{n}}} \nonumber\\
		= &  \sum_{m=0}^{{n \choose 2}-1}\left(1- \frac{n+2m}{n^2}\right)^{-1} \E{\log \pran{1-\frac{\chi(G^{(n)}_{m})}{n}} \cdot \pran{1-\frac{\chi(G^{(n)}_{m})}{n}}}\, .\label{eq:logsum_rewrite}
\end{align}
The proposition now follows from (\ref{eq:logpart_lowerbound_suscept}). 
\end{proof}
It turns out that there is a deterministic, increasing function $f:[0,\infty) \to [0,1]$ such that $\sup_{0 \le m < {n \choose 2}} |\chi(G_m^{(n)})/n-f(m/n)| \to 0$ in probability, as $n \to \infty$. Much of the rest of the paper is devoted to explaining this fact in more detail. 
However, imagine for the moment that such a function $f$ exists and, moreover, that terms in the sum with $m \gg n$ have an insignificant total contribution. With these assumptions, the sum 
in (\ref{eq:logsum_rewrite}) 
looks like a Riemann approximation for 
$\int_0^{\infty} (1-f(x)) \log(1-f(x)) \mathrm{d}x$ with mesh $1/n$. We should then expect that
\[
\E{\log\hat{Z}_{\textsc{mc}}^{\to}(n)} = 2(n-1)\log n - (1+o(1)) n\cdot \int_0^\infty (1-f(x))\log(1-f(x)) \mathrm{d}x\, .
\]
This is indeed the case. Furthermore, enough is known about $f$ that explicit evaluation of the integral is possible, and we obtain the following theorem. 
\begin{thm}\label{thm:zmc_lower}
Let 
\begin{equation}\label{zmczeta:zmcint_result}
\zeta_{\textsc{mc}} = 
\zeta(2)-3+\log 2-\log^2 2\, .
\end{equation}
Then 
\[\E{\log\hat{Z}_{\textsc{mc}}(n)} = n\cdot (2\log n + \zeta_{\textsc{mc}}+o(1)).\]
\end{thm}
Numerically, $\zeta_{\textsc{mc}}$ is around $-1.14237$. 
\begin{cor}\label{cor:exp_decay}
There is $c > 0$ such that $\p{\hat{Z}_{\textsc{mc}}(n)/\e\hat{Z}_{\textsc{mc}}(n) < e^{-cn}} \to 1$ as $n \to \infty$. 
\end{cor}
\begin{proof}
Fix $c \in \R$ and suppose that $\p{\hat{Z}_{\textsc{mc}}(n) \ge n^{2n} e^{cn}} > \eps > 0$. Then 
\begin{align*}
& \E{\log\hat{Z}_{\textsc{mc}}(n) - 2n\log n-cn+1} \\
 \ge & \eps \E{\log\hat{Z}_{\textsc{mc}}(n) - 2n\log n+cn+1~|~\hat{Z}_{\textsc{mc}}(n) \ge n^{2n} e^{-cn}} \\
									\ge & \eps. 
\end{align*}
Thus, if $\liminf_{n \to \infty}\p{\hat{Z}_{\textsc{mc}}(n) \ge n^{2n} e^{cn}} > 0$ then for all $n$ large enough, 
\[
\E{\log\hat{Z}_{\textsc{mc}}(n)} > 2n\log n+cn-1\, .
\]
It thus follows from Theorem~\ref{thm:zmc_lower} that for all $\eps > 0$, 
\[
\p{\hat{Z}_{\textsc{mc}}^{\to}(n) \ge n^{2n} e^{(\zeta_{\textsc{mc}}^{\to}+\eps)n}} \to 0
\]
as $n \to \infty$. On the other hand, 
\[
\e{\hat{Z}_{\textsc{mc}}^{\to}(n)}=n^{n-2}(n-1)! = n^{2n}e^{-n(1+o(1))}\, .
\] 
Since $\zeta_{\textsc{mc}} < -1$, the result follows. 
\end{proof} 
The form of the constant $\zeta_{\textsc{mc}}$ is unimportant, though intriguing. 
What is clear from the above is that information about the susceptibility process of the multiplicative coalescent immediately yields control on for $\hat{Z}_{\textsc{mc}}(n)$. The aim of the next section is thus to understand the susceptibility process in more detail. 

\secsub{Bounding $\chi$ using a graph exploration}\label{sec:suscbound}
\addtocontents{toc}{\SkipTocEntry}
The coupling between the ``Version 2'' multiplicative coalescent (Kruskal's algorithm) and the Erd\H{o}s-R\'enyi coalescent from Section~\ref{sec:mult_coal_other_versions} applied to arbitrary exchangeable, distinct edge weights $\mathbf{W}$. In this coupling, for $m \in [{n \choose 2}]$, we took $G^{(n)}_m$ to be the subgraph of $K_n$ consisting of the $m$ edges of smallest $\mathbf{W}$-weight. 

In the current section, it is useful to be more specific. We suppose the entries of $\mathbf{W}$ are iid Uniform$[0,1]$ random variables. 
Write $G(n,p)$ for the graph with vertices $[n]$ and edges $\{e_j: W_{e_j} \le p\}$. In $G(n,p)$, each edge of $K_n$ is independently present with probability $p$. Furthermore, we have $G(n,p)=G^{(n)}_{m_p}$, where $m_p=\max\{i: W_{e_i} \le p\}$, so this also couples the Erd\H{o}s-R\'enyi coalescent with the process $(G(n,p),0 \le p \le 1)$. The next exercise is standard, but important. 
\bsex \label{ex:gnp_cond_dist}
Show that for any $p \in (0,1)$ and $m \in {n \choose 2}$, given that $|E(G(n,p))| = m$, the conditional distribution of $G(n,p)$ is the same as that of $G^{(n)}_m$. 
\esex
For $c > 0$, let $\alpha=\alpha(c)$ be the largest real solution of $e^{-c x}=1-x$. The aim of this section is to prove the following result. 
\begin{thm}\label{thm:susc_conc}
For all $n \ge 1$ and $0 \le p \le n^{-19/20}$, 
\[
\p{|\chi(G(n,p)) - \alpha(np)^2 n| > 22n^{4/5}} < 6ne^{-12n^{1/10}}
\]
\end{thm}
The coupling with the Erd\H{o}s-R\'enyi coalescent will allow us to derive corresponding results for $G^{(n)}_m$. While the ingredients for the proof are all in the literature, and closely related results have certainly appeared in many places, we were unable to find a reference for the form we require.  Some of the basic calculations required for the proof appear as exercises; the first such exercise relates to properties of the function $\alpha$. 
\bsex \label{ex:astar_info}
\begin{itemize} 
\item[(a)] Show that $\q$ is continuous and that $\q$ is concave and strictly positive on $(1,\infty)$. 
\item[(b)] Show that for $0 < c \le 1$, $\alpha(c)=0$, and for $c \ge 2$, $1-2e^{-c} \le \alpha(c)  \le 1-e^{-c}$. 
\item[(c)] Show that $\alpha(c)$ is decreasing and $c(1-\alpha(c))$ is decreasing. 
\item[(d)] Show that $\frac{\mathrm{d}}{\mathrm{d}c} \alpha(c) \uparrow 2$ as $c \downarrow 1$. Conclude that 
$2\eps(1- o(1)) \le \alpha(1+\eps) \le 2\eps$, the first inequality holding as $\eps \downarrow 0$.
\item[(e)] Show that $\alpha(c)$ is the survival probability of a Poisson$(c)$ branching process. {\it (This exercise is not used directly.)}
\end{itemize}
\esex

Our proof of Theorem~\ref{thm:susc_conc} hinges on a variant of the well-known and well-used depth-first search exploration procedure. In depth-first search, at each step one vertex is ``explored'': its neighbours are revealed, and those neighbours lying in the undiscovered region of the graph are added to the ``depth-first search queue'' for later exploration. 
In our variant, if the queue is ever empty, in the next step we add each undiscovered vertex to the queue independently with probability $p$. (It is more standard to add a {\em single} undiscovered vertex, but adding randomness turns out to simplify the formula for the expected number of unexplored vertices.) 

We now formally state our search procedure for $G(n,p)$. At step $i$ the vertex set $[n]$ is partitioned into sets $E_i,D_i$ and $U_i$, respectively containing {\em explored, discovered}, and {\em undiscovered} vertices. We always begin with $E_0=\emptyset$, $D_0=\{1\}$, and $U_0=[n]\setminus\{1\}$. 
For a set $S$, we write $\mathrm{Bin}(S,p)$ to denote a random subset of $S$ which contains each element of $S$ independently with probability $p$. For $v \in [n]$ we write $N(v)$ for the neighbours of $v$ in $G(n,p)$. Finally, we define the {\em priority} of a vertex $v \in [n]$ is its time of discovery $\inf\{j: v \in D_j\}$, so vertices that are discovered later have higher priority. 
\begin{mdframed}[style=algorithm]
\noindent {\bf Search process for $G(n,p)$.}
 \\
 {\sc Step }i: \\ 
$\star$~If $D_i\ne \emptyset$ then choose $v \in D_i$ with highest priority (if there is a tie, pick the vertex with smallest label among highest-priority vertices). Let $E_{i+1}=E_i \cup \{v\}$, let $D_{i+1}=(D_i \cup (N(v) \cap U_i))\setminus \{v\}$ and let $U_{i+1}=U_i\setminus(N(v) \cap U_i)$. \\
$\star$~If $D_i = \emptyset$ then let $D_{i+1}=\mathrm{Bin}(U_i,p)$, independently of all previous steps. Let $E_{i+1}=E_i$ and let $U_{i+1}=U_i\setminus D_{i+1}$. 
\end{mdframed}

\medskip
Observe that the sequence $((D_i,E_i,U_i),i \ge 0)$ describing the process may be recovered from either $(D_i,i \ge 0)$ or $(U_i,i \ge 0)$. 
The order of exploration yields the following property of the search process. Suppose $D_i=\emptyset$ for a given $i$. Then $D_{i+1}$ may contain several nodes, all of which have priority $(i+1)$. Starting at step $(i+1)$, the search process will fully explore the component containing the smallest labelled vertex of $D_{i+1}$ before exploring any vertex in any other component. More strongly, the search process will explore the components that intersect $D_{i+1}$ in order of their smallest labeled vertices. 

For $i > 0$ such that $E_{i} \ne E_{i-1}$, write $v_i$ for the unique element of $E_i\setminus E_{i-1}$. Say that a component exploration concludes at time $t$ if $v_{t+1}$ and $v_t$ are in distinct components of $G(n,p)$. The observation of the preceding paragraph implies the following fact about the search process. Set $D_0=\emptyset$ for convenience. 
\begin{fact}\label{fac:d_size}
Fix $t>0$ and let $i=i(t)=\max\{j<t: D_j=\emptyset\}$. If a component exploration concludes at time $t$ then $|D_{i+1}| \ge n-t-|U_t|$. 
\end{fact}
\begin{proof}
Since a component exploration concludes at time $t$ we have $D_t \subset D_{i+1}$. Furthermore, $|E_t| \le t$ because $|E_0|=\emptyset$ and $|E_{j+1}\setminus E_j| \le 1$ for all $j \ge 0$. As $D_t$, $E_t$ and $U_t$ partition $[n]$, we thus have 
\[
|U_t| =n-|E_t| - |D_t| \ge n-t-|D_{i+1}|.\qedhere
\]
\end{proof}

In proving Theorem~\ref{thm:susc_conc} we use a concentration inequality due to \citet{mcdiarmid89method}. 
Let $\mathrm{X}=(X_i,1 \le i \le m)$ be independent Bernoulli$(q)$ random variables. Suppose that 
$f:\mathbb\{0,1\}^m$ is such that for all $1 \le k \le m$, for all $(x_1,\ldots,x_k) \in \{0,1\}^k$, 
\[
|\E{f(x_1,\ldots,x_k,X_{k+1},\ldots,X_m)} - \E{f(x_1,\ldots,1-x_k,X_{k+1},\ldots,X_m)}|.
\] 
In other words, given the values of the first $k-1$ variables, knowledge of the $k$'th variable changes the conditional expectation by at most one. 
\begin{thm}[McDiarmid's inequality]\label{thm:mcd_ineq}
Let $\mathrm{X}$ and $f$ be as above. Write $\mu=\E{f(\mathrm{X})}$. Then for $x > 0$, 
\[
\p{f(X) \ge \mu + x} \le e^{-x^2/(2mq + 2x/3)},\quad \p{f(X) \le \mu - t} \le e^{-x^2/(2mq + 2x/3)}\, .
\]
\end{thm}

Our probabilistic analysis of the search process begins with the following observation. For each $i \ge 0$, the set $U_{i}\setminus U_{i+1}=D_{i+1}\setminus D_i$ of vertices discovered at step $i$ has law $\mathrm{Bin}(U_i,p)$. 
This observation also allows us to couple the search process with a family $\mathrm{B}=(B_{i,j},i \ge 1,j \ge 1)$ of iid Bernoulli$(p)$ random variables, by inductively letting $U_{i}\setminus U_{i+1} = D_{i+1}\setminus D_{i}$ equal $\{j \in U_{i}: B_{i,j}=1\}$, for each $i \ge 1$. 
The coupling shows that for all $i \ge 1$, $U_i$ satisfies the hypotheses of Theorem~\ref{thm:mcd_ineq}, with $m=ni$ and $q=p$. 
Also, using the preceding coupling, the next exercise is an easy calculation. 
\bsex \label{ex:poisson_identity}
Show that for $i\ge 0$, $\E{|U_{i+1}|~|~(U_j,j \le i)} = |U_{i}|(1-p)$; conclude that $\e |U_i|=(n-1)(1-p)^{i}$ for all $i\ge 0$. 
\esex
The exploration of the component $C(1)$ is completed precisely at the first time $j$ that $D_j=\emptyset$; this is also the first time $j$ that $|U_j|= n-j$, and for earlier times $k$ we have $U_k < n-k$. If we had $|U_i|=\E|U_i|$ for all $i$ then the above exercise would imply that $|C(1)| = \min\{t \in \N: (n-1)(1-p)^t \ge n-t\}$. Of course, $|U_i|$ does not equal $\e|U_i|$ for all $i$. However, $|U_i|$ does track its expectation closely enough that a consideration of the expectation yields an accurate prediction of the first-order behaviour of $\chi(G(n,p))$. We next explain this in more detail, then proceed to the proof of Theorem~\ref{thm:susc_conc}. Write $t(n,p)$ for the largest real solution of $n(1-p)^x=n-x$. We will use the next exercise, the first part of which which gives an idea of how $t(n,p)$ behaves when $p$ is moderately small. 
\bsex \label{ex:tstar_info}
\begin{itemize} 
\item[(a)] Show that $t(n,p) = n \cdot \alpha(n \log(1/(1-p)))$. Conclude that if $p \le n^{-3/4}$ then with $c=np$, the largest real solution $t=t(n,p)$ of $n(1-p)^x = n-x$ satisfies 
\[
\alpha(c) n \le t \le \alpha(c) n + \frac{2n^{1/2}}{1-p}.
\] 
({\bf Hint.} Use Exercise~\ref{ex:astar_info}~(d).)
\item[(b)] 
Show that 
\begin{align*}
n(1-p)^s & \ge (n-s) + (s-t)(1+(n-t)\log(1-p))\mbox{ for }s > t\, 
\end{align*}
\end{itemize}
\esex
Write $L$ and $S$ for the sizes of the largest and second largest components of $G(n,p)$, respectively. From time $0$ to time $t=t(n,p)$, the search process essentially explores a single component. We thus expect that $L \ge t$. Next, since $n(1-p)^{t+1} > n-(t+1)$ and $n(1-p)^t = n-t$, by the convexity of $(1-p)^s$ we have $n(1-p)^{s+1} \ge n(1-p)^s-1$ for all $s \ge t$. Exercise~\ref{ex:poisson_identity} then implies that $\e|U_{s+1}| \ge \e|U_s|-1$ for all integer $s \ge t$. In other words, when exploring a component after time $t$, the search process on average discovers less than one new vertex in each step. Such an exploration should quickly die out and, indeed, after time $t$ the components uncovered by the search process typically all have size $o(n)$. Together with the first point, this suggests that $L \le t+o(n)$ and $S=o(n)$. Using the bounds on $t$ from Exercise~\ref{ex:tstar_info}~(a) and the bounds on $\chi(G)$ from Exercise~\ref{eq:susc_comp_size_bounds}, we are led to predict that 
\[
\alpha(np)^2 n + o(n) = \frac{L^2}{n} \le \chi(G(n,p)) \le \frac{L^2}{n} + S = \alpha(np)^2 n + o(n). 
\]
Theorem~\ref{thm:susc_conc} formalizes and sharpens this prediction, and we now proceed to its proof. 
\begin{proof}[Proof of Theorem~\ref{thm:susc_conc}] 
Throughout the proof we assume $n$ is large (which is required for some of the inequalities), and write $t=t(n,p)$, $\alpha=\alpha(np)$. 

\medskip
\noindent \fbox{{\bf Case 1: $p \le 1/n+6/n^{6/5}$ (``subcritical $p$'').}}\\

Recall that exploration of $C(1)$ concludes the first time $i$ that $|U_i| \ge n-i$. 
Letting $t^+=21 n^{4/5}$, we have $(1-p)^{t^+} \ge 1-{t^+}p+({t^+}p)^2/2-({t^+}p)^3/6> 1-{t^+}p+({t^+}p)^2/3$, and it follows straightforwardly that 
\[ 
\e{|U_{t^+}|}= n(1-p)^{t^+} \ge n(1-{t^+}p + ({t^+}p)^2/3) \ge n-{t^+} + 3 n^{3/5}. 
\] 
Applying the lower bound from Theorem~\ref{thm:mcd_ineq} to $|U_t|$, it follows that 
\[ 
\p{|U_{t^+}| \le n-{t^+}} \le \p{|U_{t^+}| \le \e|U_{t^+}|-3n^{3/5}} \le e^{-(9/2)n^{1/5}}\, . 
\] 
At all times $i$ before exploration of the first component concludes we have $|U_i| < n-i$, so the preceding bound yields 
\[ 
\p{|C(1)| \ge 21 n^{4/5}} \le e^{-(9/2)n^{1/5}},  
\] 
We always have $\chi(G(n,p)) \le \max_{i \in [n]} |C(i)|$ so, by a union bound, 
\[
\p{\chi(G(n,p)) \ge 21n^{4/5}} \le \p{\max_{i \in [n]} |C(i)| \ge 21n^{4/5}} \le ne^{-(9/2)n^{1/5}}\, .
\]
For this range of $p$ we also have $n\alpha(np) \le 12n^{4/5}$, and so the bound in Theorem~\ref{thm:susc_conc} follows. 

\medskip
\noindent\fbox{\noindent {\bf Case 2: $1/n + 6/n^{6/5} < p \le 1/n^{19/20}$ (``supercritical $p$'').}}\\

We begin by explaining the steps of the proof. {\sc (I)} First, logic similar to that in case $1$ shows that the largest component of $G(n,p)$ is unlikely to have size much larger than $t$. {\sc (II)} Next, we need a corresponding {\em lower} tail bound on the size of the largest component; the proof of this relies on Fact~\ref{fac:d_size}. {\sc (III)} Finally, we need to know that with high probability there is only one component of large size; after ruling out one or two potential pathologies, this follows from the subcritical case. We treat the three steps in this order. Write $\Delta=n^{3/4}$ and $t^{\pm}=t(n,p)\pm \Delta$. 

\medskip

\noindent {\sc (I)} 
We claim that 
\begin{equation}\label{eq:nptplusbound}
n(1-p)^{t^+} \ge n-t^+ + 5n^{11/20}. 
\end{equation}
To see this, first use Exercise~\ref{ex:tstar_info}~(b) to obtain 
\[
n(1-p)^{t^+} \ge n-t^+ + \Delta (1+(n-t)\log(1-p))\, .
\]
Let $c=n \log(1/(1-p))$. By Exercise~\ref{ex:tstar_info}~(a), 
\[1+(n-t)\log(1-p) = 1-c(1-\alpha(c))\, .
\]
Next, as $p \ge 1/n+6/n^{6/5}$ we have $n \log(1/1-p) \ge np \ge 1+6/n^{1/5} =: c^*$. 
By Exercise~\ref{ex:astar_info}~(c) and (d), 
it follows that 
\[
c(1-\alpha(c)) \le c^*(1-\alpha(c^*)) = \left(1+\frac{6}{n^{1/5}})\right)\left(1-\frac{(2+o(1))6}{n^{1/5}}\right) \le 1-\frac{5}{n^{1/5}}\, ,
\]
so $1+(n-t)\log(1-p) \ge 5/n^{1/5}$. (Similar bounds using Exercise~\ref{ex:astar_info}~(c) and (d) crop up again later in the proof). 
Since $\Delta/n^{1/5}=n^{11/20}$, (\ref{eq:nptplusbound}) follows. 
Having established (\ref{eq:nptplusbound}), essentially the same logic as in Case 1 yields 
\begin{equation}\label{eq:tplus_bound}
\p{\max_{i \in [n]} |C(i)| \ge t^+} \le n\p{|C(1)| \ge t^+} \le ne^{-(25/2) n^{1/10}}\, .
\end{equation}

\medskip

\noindent {\sc (II)} We now turn to the lower tail of $\max_{i \in [n]} |C(i)|$. 
The calculations are similar but slightly more involved. 
Since $p=o(1)$ and $p\Delta \le n^{-1/5}=o(1)$, for $n$ large $(1-p)^{-\Delta} \le 1+p\Delta+(p\Delta)^2$, so 
\begin{equation}\label{eq:ntminus_bound}
n(1-p)^t{^-} = (n-t)(1-p)^{-\Delta} \le (n-t)(1+p\Delta+(p\Delta)^2). 
\end{equation}
Since $t \ge n\alpha-n^{1/2}$, it follows easily from Exercises~\ref{ex:astar_info}~(c) and~(d) that $p(n-t) \le 1-\frac{5}{n^{1/5}}$. 
Using (\ref{eq:ntminus_bound}) and the bound $p\Delta \le n^{-1/5}$, we thus have 
\begin{align}
n(1-p)^{t^-} & 
\le (n-t) + \Delta\left(1-\frac{5}{n^{1/5}}\right)(1+p\Delta) \nonumber\\
& \le n-t^- - \frac{4\Delta}{n^{1/5}} \nonumber\\
& = n-t^--4n^{3/5}. \label{eq:tminus_bound} 
\end{align}

Next, basic arithmetic shows that if $m \ge (n+p^{-1})/2$ then $m(1-p) \le m-1-(np-1)/2 \le m-(1+6n^{-1/5})$. 
Furthermore, for $p$ in the range under consideration, $(n+p^{-1})/2 \le n-2n^{4/5}$, so 
\[
n(1-p)^{\Delta} \le \max(n-2n^{4/5},n-\Delta - 6\Delta n^{-1/5}) =  n-\Delta-6n^{11/20}\, .
\]
Since $n(1-p)^t$ is concave as a function of $t$, this bound and (\ref{eq:tminus_bound}) together imply that 
$n(1-p)^x \le n-x-6n^{11/20}$ for all $x \in [\Delta,t^-]$. Applying Theorem~\ref{thm:mcd_ineq} for $t \in [\Delta,t^-]$, and a union bound, yields
\[
\p{\exists t \in [\Delta,t^-]: |U_t| \ge n-t-n^{11/20}} \le (t^--\Delta)e^{-(25/2)n^{1/10}}\, .
\]

Now suppose that $|U_t| < n-t-n^{11/20}$ for all $t \in [\Delta,t^-]$. In this case, if a component exploration concludes at some time $t \in [\Delta,t^-]$ then by Fact~\ref{fac:d_size} there is $i < t$ such that $D_i=\emptyset$ and $|D_{i+1}| > n-t-|U_t| > n^{11/20}$. On the other hand, for all $i \ge 0$, $|D_{i+1}\setminus D_i|$ is stochastically dominated by $\mathrm{Bin}(n,p)$, so by a union bound followed by a Chernoff bound (or an application of Theorem~\ref{thm:mcd_ineq}),  
\[
\p{\exists i < t^-: D_i=\emptyset,|D_{i+1}| > n^{11/20}} \le t^- \p{\mathrm{Bin}(n,p) > n^{11/20}} \le t^- e^{-n^{11/20}}\, .
\]
It follows that 
\begin{align}
& \p{\mbox{A component exploration concludes between times $\Delta$ and $t^-$}}\nonumber \\
&\le (t^--\Delta)e^{-(25/2)n^{1/10}} + t^- e^{-n^{11/20}} \nonumber\\
& \le 2ne^{-(25/2)n^{1/10}}\, . \label{eq:supercrit_lower_bound}
\end{align}

\medskip
\noindent {\sc (III)} 
Let $N$ be the number of vertices remaining when the first time after time $t^-$ that the search process finishes exploring a component, 
and write $B$ for the event that some component whose exploration starts {\em after} time $t^-$ has size greater than $21n^{4/5}$. 
Then 
\begin{align*}
\p{B} & \le \p{B, N > n-(t^--\Delta)} + \sum_{m \le n-(t^--\Delta)} \p{B,N=m} \\
 \le & \p{N > n-(t^--\Delta)} + \sup_{m \le n-(t^--\Delta)} \p{B~|~N=m}\, .
\end{align*}
The first probability is at most $2ne^{-(25/2)n^{1/10}}$ by (\ref{eq:supercrit_lower_bound}). 
To bound the second, note that 
\[
n-(t^--\Delta) \le n+2n^{3/4}-t \le n+3n^{3/4}-n\alpha(np) = n(1-\alpha(np)+3n^{-1/4}). 
\]
By Exercise~\ref{ex:astar_info}~(c) and~(d), and since $p \le n^{-19/20}$, for $m \le n-(t^--\Delta)$ we therefore have 
\[
mp \le np(1-\alpha(np))+3n^{3/4}p < 1.
\]
For such $m$, the bound for ``subcritical $p$'' from Case 1 thus yields
\[
\p{B~|~N=m} \le me^{-(9/2) m^{1/5}}. 
\]
This is less than $ne^{-(9/2) n^{4/25}}$ for $m \ge n^{4/5}$. If $m \le n^{4/5}$ then the largest component explored after time $m$ also has size $\le n^{4/5}$,  so $\p{B~|~N=m}=0$. We conclude that 
\begin{equation}\label{eq:bound_after_giant_done}
\p{B}\le 2ne^{-(25/2)n^{1/10}} + ne^{-(9/2) n^{4/25}} \le 3ne^{-(25/2)n^{1/10}}\, .
\end{equation}

\medskip
\noindent {\sc (IV)} Now to put the the pieces together. The lower bound is easier: by (\ref{eq:supercrit_lower_bound}) and the first inequality from Exercise~\ref{eq:susc_comp_size_bounds}, 
inequality $\chi(G(n,p)) \ge n^{-1} \max_{i \in [n]} |C(i)|^2$, 
\begin{equation}\label{eq:thmbound_lower}
\p{\chi(G(n,p)) < \frac{(t^--\Delta)^2}{n}} \le 2ne^{-(25/2)n^{1/10}}\, , 
\end{equation}
and by Exercise~\ref{ex:tstar_info}~(a), 
\[
\frac{(t^--\Delta)^2}{n} = \frac{(t-2n^{3/4})^2}{n} \ge \frac{(n\alpha-3n^{3/4})^2}{n} \ge n\alpha^2 - 9n^{1/2}\,. 
\]

For the upper bound, any component of $G(n,p)$ whose exploration concludes before step $n^{3/4}$ of the search process has size at most $n^{3/4}$. Write $S$ for the number of vertices of the second-largest component of $G(n,p)$. By (\ref{eq:bound_after_giant_done}), we then have 
\[
\p{S \ge 21 n^{4/5}} \le 3ne^{-(25/2)n^{1/10}}\, .
\]
Combined with the second inequality from Exercise~\ref{eq:susc_comp_size_bounds} 
and with (\ref{eq:tplus_bound}), we obtain 
\begin{equation}\label{eq:thmbound_upper}
\p{\chi(G(n,p)) \ge \frac{(t+n^{3/4})^2}{n} + 21n^{4/5}} \le 4ne^{-(25/2)n^{1/10}}. 
\end{equation}
An easy calculation using Exercise~\ref{ex:tstar_info}~(a) shows that $(t+n^{3/4})^2/n + 21 n^{4/5} \le n\alpha^2 + 22 n^{4/5}$, and the theorem then follows from (\ref{eq:thmbound_lower}) and (\ref{eq:thmbound_upper}). 
\end{proof}
To conclude the section, we use Theorem~\ref{thm:susc_conc} to show that $\E{\chi(G^{(n)}_m)}$ is well-approximated by $\alpha(2m/n)$ in a range which covers the most important values of $m$. (Exercise~\ref{ex:connect_threshold}, below, extends this to all $0 \le m \le {n \choose 2}$.)
\begin{lem}\label{lem:susc_exp_conc}
For $n$ large, for all $m \le n^{10/9}$, 
\[\left|\E{\chi(G^{(n)}_m)} - \alpha^2(2m/n)n\right| \le 23 n^{4/5}.
\] 
\end{lem}
\begin{proof}
Write $x_m = \inf\{p: |E(G(n,p))|=m\}$. Since $\e{|E(G(n,p))|}=p{n \choose 2}$, we expect $x_m$ to be near $p_m := m/{n \choose 2}$. Write $\hat{\alpha}=\alpha(2m/(n-1))=\alpha(np_m)$, let $\delta = n^{-4/3}$, and let $p_m^{\pm}=p_m \pm \delta$. 

In the coupling of $(G(n,p),0 \le p \le 1)$ and $(G^{(n)}_m, 0 \le m \le {n \choose 2})$, if $x_m > p_m^-$ then $G(n,p_m^-)$ is a subgraph of $G^{(n)}_m$ and so $\chi(G^{(n)}_m) \ge \chi(G(n,p_m^-))$. Likewise, 
if $x_m < p_m^+$ then $\chi(G^{(n)}_m) \le \chi(G(n,p_m^.))$
We thus have 
\begin{align*}
\chi(G^{(n)}_m) & \ge \chi(G(n,p_m^-)) \I{x_m > p_m^-} \\
			& \ge \chi(G(n,p_m^-)) - n \I{x_m \le p_m^-}, \quad \mbox{ and}\\
\chi(G^{(n)}_m) & \le \chi(G(n,p_m^+)) \I{x_m < p_m^+} + n\I{x_m \ge p_m^+} \\
			& \le \chi(G(n,p_m^+)) + n\I{x_m \ge p_m^+}. 
\end{align*} 
Since $\alpha$ is $2$-Lipschitz, $\hat{\alpha}-2/n^{1/3} \le \alpha (np_m^-) \le \alpha(np_m^+) \le \hat{\alpha} + 2/n^{1/3}$, from which it follows that 
both $\alpha(np_m^-)^2n$ and $\alpha(np_m^+)^2n$ are within $5 n^{2/3}$ of $\hat{\alpha}^2 n$. By the preceding lower bound on $\chi(G^{(n)}_m)$ and Theorem~\ref{thm:susc_conc} we thus have 
\begin{align*}
\E{\chi(G^{(n)}_m)} & \ge \E{\chi(G(n,p_m^-))} - n \P{x_m \le p_m^-} \\
			& \ge \hat{\alpha}^2 n -5 n^{2/3}-22n^{4/5} - n \P{\mathrm{Bin}\left({n \choose 2},p_m^-\right) \ge m} \\
			& \ge \hat{\alpha}^2 n -5 n^{2/3}-22n^{4/5} - 1\, 
\end{align*}
the last inequality holding straightforwardly by a Chernoff bound (note that ${n \choose 2}p_m^-=m-(n-1)/(2n^{1/3}) \le m-m^{3/5}/3$). 
We likewise have 
\begin{align*}
\E{\chi(G^{(n)}_m)}	& \le \hat{\alpha}^2 n+ 5 n^{2/3}+22n^{4/5} + 1. 
\end{align*}
Finally, $2m/(n-1)-2m/n = 2m/(n(n-1)) =O(n^{-8/9})$, so since $\alpha$ is $2$-Lipschitz we have $\alpha(2m/n)^2 =\hat{\alpha}^2+O(n^{-8/9})$, and the result follows. 
\end{proof}
\iftoggle{bookversion}{
In the next section, ... Frieze limit

\secsub{Frieze's $\zeta(3)$ limit for the MST weight}}
{
\section{Frieze's $\zeta(3)$ limit for the MST weight}}

\label{sec:frieze_lim}
Before proving Theorem~\ref{thm:zmc_lower}, we warm up by using the same approach to study the total weight of random MSTs. Throughout the section,  $\mathbf{W}=(W_e,e \in E(K_n))$ are exchangeable, distinct, non-negative edge weights. Recall from Section~\ref{sec:mult_coal_other_versions} that ``Version 2'' of the multiplicative coalescent (aka Kruskal's algorithm) considers edges one-by-one in increasing order of weight, adding only edges which connect distinct trees in the forest, and that the result is the minimum spanning tree $T=\mathrm{MST}(K_n,\mathbf{W})$. 

Write $w(T) = \sum_{e \in E(T)} W_e$ for the total weight of $T$. We use susceptibility bounds to approximate $w(T)$ and derive a version of Frieze's famous $\zeta(3)$ limit. 
\begin{thm}[\citet{frieze85mst}]\label{thm:frieze}
Write $X_1,\ldots,X_{n \choose 2}$ for the increasing ordering of $\mathbf{W}$. If $\e{X_m}=(1+o(1))m$, then $2\e{w(MST(K_n,\mathbf{W}))}/n^2 \to \zeta(3)$ as $n \to \infty$. 
\end{thm}
By $\e{X_m}=(1+o(1))m$ we mean that $\lim_{n \to \infty} \sup_{m \in [{n\choose 2}]} |1-\e X_m/m|=0$. This condition can be relaxed, and the proof can be modified to obtain convergence in probability under suitable hypotheses, but for exporitory reasons we have opted for simplicity over full generality. 
Before beginning the proof, we first note a special case. Suppose that the weights $W_e$ are independent Uniform$[0,1]$ random variables. Then 
$\E{X_k} = k/{n \choose 2}$, $\E{X_k\cdot n^2/2} = (1+o(1))k$. The theorem thus implies that for such uniform edge weights, the toal weight of the random MST of $K_n$ converges to $\zeta(3)$ without renormalization. This is the most often quoted special case of Frieze's result. 

Our proof is based on the following identity for $\E{w(T)}$.  
\begin{prop}\label{prop:mst_weight_formula}
Write $X_1,\ldots,X_{n \choose 2}$ for the increasing ordering of $\mathbf{W}$. Then 
\begin{equation}\label{eq:mst_weight_identity}
\E{w(T)} = \sum_{m=0}^{{n \choose 2}-1} \e X_{m+1} \cdot \p{\chi(G^{(n)}_{m+1}) > \chi(G^{(n)}_{m})}. 
\end{equation}
\end{prop}
\begin{proof}
Let $e_1,\ldots,e_{n \choose 2}$ be the ordering of $E(K_n)$ by increasing weight, so $e_m$ has weight $X_m$. 
In the coupling with the Erd\H{o}s-R\'enyi coalescent, Kruskal's algorithm adds edge $e_k$ precisely if $e_k$ joins distinct components of $G^{(n)}_{k-1}$, which occurs if and only if $\chi(G^{(n)}_k) > \chi(G^{(n)}_{k-1})$. For this coupling we thus have 
\[
w(T) = \sum_{m=0}^{{n \choose 2}-1} X_{m+1} \cdot \I{\chi(G^{(n)}_{m+1}) > \chi(G^{(n)}_{m})}. 
\]
By the exchangeability of $\mathbf{W}$, the vector $(X_1,\ldots,X_{n \choose 2})$ is independent of the ordering of $E(K_n)$. The event that $\chi(G^{(n)}_{m+1}) > \chi(G^{(n)}_{m})$ is measurable with respect to the ordering of $E(K_n)$, so is independent of $(X_1,\ldots,X_{n \choose 2})$. The proposition follows on taking expectations. 
\end{proof}
We use the result of the following exercise to deduce that terms with $m \ge 5n\log n$ play an unimportant role in the summation (\ref{eq:mst_weight_identity}). Fix $1 \le k \le \lfloor n/2\rfloor$ and let $N_k$ be the number of sets $A \subset [n]$ such that, in $G^{(n)}_m$, there are no edges from $A$ to $[n\setminus A]$. Note that $G^{(n)}_m$ is connected precisely if $N_k=0$ for all $1 \le k \le {n\choose 2}$. 
\bsex
\begin{itemize}\label{ex:connect_threshold}
\item[(a)] Let $E_k$ be the event that there are no edges from $[k]$ to $[n]\setminus[k]$. With $p=m/{n\choose 2}$, show that 
$\p{E_k} \le (1-p)^{k(n-k)}$. Deduce that 
\[
\p{N_k > 0} \le n^k(1-p)^{k(n-k)} \le (ne^{-p(n-k)})^k\, .
\]
\item[(b)] Show that $\p{G^{(n)}_{\lceil 5n\log n \rceil}\mbox{ is not connected}} \le n^{-4}$. 
\item[(c)] Show that the bound in Lemma~\ref{lem:susc_exp_conc} in fact holds for all $m \in [{n \choose 2}]$. 
\end{itemize}
\esex
\begin{cor}\label{cor:mst_weight_formulae}
With the notation of Proposition~\ref{prop:mst_weight_formula}, 
we have 
\begin{align*}
\E{w(T)} & \ge \sum_{m=0}^{5n\log n} \e X_{m+1} \left(1-\frac{\e \chi(G^{(n)}_{m})}{n}\right)\, \quad \mbox{ and} \\
\E{w(T)} & \le \left(1+\frac{12\log n}{n}\right) \sum_{m=0}^{5n\log n} \e X_{m+1} \left(1-\frac{\e \chi(G^{(n)}_{m})}{n}\right) + \frac{1}{2n^2}\E{X_{n\choose 2}}\, .
\end{align*}
\end{cor}
\begin{proof}
Write 
\[
\p{\chi(G^{(n)}_{m+1}) > \chi(G^{(n)}_{m})} = \E{\p{\chi(G^{(n)}_{m+1}) > \chi(G^{(n)}_{m})~|~G^{(n)}_{m}}}.
\]
We derived an identity for the inner conditional probability in (\ref{eq:prob_to_exp_susc}); using that identity and linearity of expectation, we obtain 
\begin{equation}\label{eq:prob_chi_increase}
\p{\chi(G^{(n)}_{m+1}) > \chi(G^{(n)}_{m})} = \left(1-\frac{\e \chi(G^{(n)}_{m})}{n}\right)\left(1- \frac{n+2m}{n^2}\right)^{-1}. 
\end{equation}
The latter is always at least $1-\e \chi(G^{(n)}_{m})/n$, and the lower bound then follows from Proposition~\ref{prop:mst_weight_formula} by truncating the sum at $m=5n\log n$. 

For the upper bound, note that 
if $G^{(n)}_m$ is connected then $\chi(G^{(n)}_m)=n$, so 
\[
\p{\chi(G^{(n)}_{m+1}) > \chi(G^{(n)}_{m})} \le \p{G^{(n)}_{m})\mbox{ is not connected}}\, .
\]
Using Exercise~\ref{ex:connect_threshold}~(b) and the fact that the $X_i$ are increasing, it follows that 
\begin{align*}
\E{w(T)} & \le \sum_{m=0}^{5 n \log n} \e X_{m+1} \cdot \p{\chi(G^{(n)}_{m+1}) > \chi(G^{(n)}_{m})} \\
		& + \sum_{m=5 n \log n+1}^{{n \choose 2}-1} \e X_{m+1} \cdot \p{G^{(n)}_m\mbox{ is not connected}}\, \\
		& \le \sum_{m=0}^{5 n \log n} \e X_{m+1} \cdot \p{\chi(G^{(n)}_{m+1}) > \chi(G^{(n)}_{m})} + 
			{n \choose 2}\cdot \E{X_{{n \choose 2}}} \cdot \frac{1}{n^4}\, .
\end{align*}
For $m \le 5 n \log n$, $(1- \frac{n+2m}{n^2})^{-1} \le 1+12\log n/n$, and the result follows from (\ref{eq:prob_chi_increase}). 
\end{proof}
To prove Theorem~\ref{thm:frieze}, we use Lemma~\ref{lem:susc_exp_conc} and Corollary~\ref{cor:mst_weight_formulae} to show that after appropriate rescaling, the sum in Proposition~\ref{prop:mst_weight_formula} is essentially a Riemann sum approximating an appropriate integral. The value of that integral is derived in the following lemma. 
\begin{lem}\label{lem:zeta_integral}
\[
\int_0^{\infty} \lambda \cdot (1-\alpha^2(\lambda)) \mathrm{d}\lambda = 2\zeta(3)\, . 
\]
\end{lem}
\begin{proof}
\citet{aldous2004omp} write that ``calculation of this integral is quite a pleasing experience''; though the calculation appears in that work, why should we deprive ourselves of the pleasure? Anyway, the proof is short. First, use integration by parts to write 
\[
\int_0^{\infty} \lambda \cdot (1-\alpha^2(\lambda)) \mathrm{d}\lambda = \int_0^{\infty} \alpha(\lambda)\alpha'(\lambda)\lambda^2 \mathrm{d}\lambda = \int_1^{\infty} \alpha(\lambda)\alpha'(\lambda)\lambda^2 \mathrm{d}\lambda\, ,
\]
the second equality since $\alpha(\lambda)=0$ for $\lambda < 1$. The identity $1-\alpha(c)=e^{-c\alpha(c)}$ (this is how we {\em defined} $\alpha$) implies that $\lambda^2 = (\alpha(\lambda)^{-1}\log(1-\alpha(\lambda)))^2$, so we may rewrite the latter integral as 
\[
\int_{1}^{\infty} \frac{\log^2(1-\alpha(\lambda))}{\alpha(\lambda)}\cdot \alpha'(\lambda)\mathrm{d}\lambda = \int_0^1 \frac{\log^2(1-\alpha)}{\alpha}\mathrm{d}\alpha\, ,
\]
where we used the obvious change of variables $\alpha=\alpha(\lambda)$. Now a final change of variables: $u=-\log(1-\alpha)$ transforms this into 
\[
\int_0^{\infty} u^2 \frac{e^{-ku}}{1-e^{-ku}}\mathrm{d}u = \int_0^{\infty} u^2 \sum_{k=1}^{\infty} e^{-ku} \mathrm{d}u\, .
\]
Since $\int_0^{\infty} u^2 e^{-ku} = 2/k^3$, the final expression equals $\sum_{k=1}^{\infty} 2/k^3 = 2\zeta(3)$. 
\end{proof}
Our final step before the proof is to show that the integrand is well-behaved on the region of integration; the straightforward bound we require is stated in the following exercise. Recall that $\alpha$ is continuous on $[0,\infty)$ and is differentiable except at $x=1$. 
\bsex
Let $f(x)=x(1-\alpha^2(x))$, where $\alpha$ is as above. Show there exists $C<\infty$ such that $|f'(x)| \le C$ for all $x\ne 1$. (In fact we can take $C=2$.)
\esex
\begin{proof}[Proof of Theorem~\ref{thm:frieze}]
By the preceding exercise, for all $0 < \eps \le x$, 
\[
\left|\int_{x-\eps}^{x} \lambda(1-\alpha^2(\lambda))\mathrm{d}\lambda - \eps x(1-\alpha^2(x))\right| \le C\eps^2\, .
\]
Taking $\eps=2/n$, $x=2m/n$ and summing over $m \in \{1,\ldots,5n\log n\}$ we obtain in particular that 
\begin{align*}
\sum_{m=1}^{5n\log n} \frac{4m}{n^2} (1-\alpha^2(2m/n)) & = \int_0^{10\log n}\lambda(1-\alpha^2(\lambda))\mathrm{d}\lambda + O\left(\frac{\log n}{n}\right)\, \\
& = 2\zeta(3)-o(1)\, ,
\end{align*}
the second equality by Lemma~\ref{lem:zeta_integral}. 
If $\E{X_m}=(1+o(1))m$ then by the preceding equation, Lemma~\ref{lem:susc_exp_conc}, and the lower bound in Corollary~\ref{cor:mst_weight_formulae}, we have 
\[
\E{w(T)} \ge (1+o(1))\frac{n^2}{2}\cdot \zeta(3)\, ,
\]
and likewise (this time using the upper bound in Corollary~\ref{cor:mst_weight_formulae}) 
\begin{align*}
 \E{w(T)} & \le (1+o(1))\frac{n^2}{2}\cdot \zeta(3) + \frac{1}{2n^2}\E{X_{n\choose 2}}\\
 & = (1+o(1))\frac{n^2}{2}\cdot \zeta(3) + O(1)\, ,
\end{align*}
which completes the proof. 
\end{proof}

\section{Estimating the empirical entropy}\label{sec:final_proof}
We already know the broad strokes of the argument, since they are the same as for our proof of Theorem~\ref{thm:frieze}. Recall that we are trying to approximate $\E{\log \hat{Z}_{\textsc{mc}}(n)}=\E{\log \hat{Z}_{\textsc{mc}}^{\to}(n)}-(n-1)\log 2$. Proposition~\ref{prop:empirical_mc_identity} reduces this to the study of the sum 
\begin{equation}\label{eq:xidef}
\Xi = \sum_{m=0}^{{n \choose 2}- 1} \left(1-\frac{n+2m}{n^2}\right)^{-1}\E{\pran{1-\frac{\chi(G^{(n)}_{m})}{n}}\log \pran{1-\frac{\chi(G^{(n)}_{m})}{n}} }.
\end{equation}
We use Theorem~\ref{thm:susc_conc} and Exercise~\ref{ex:connect_threshold} to approximate this sum by an integral. Before proceeding to the $\eps$'s and $\delta$'s, we evaluate the integral. 
\begin{prop}\label{zmczeta:zmcinformal_prop}
We have 
\begin{equation}\label{eq:twozeta_eq}
\int_0^{\infty} (1-\q^2(\lambda))\cdot\log(1-\q^2(\lambda))\mathrm{d}\lambda=2(\zeta_{\textsc{mc}}+\log 2)\, .
\end{equation}
\end{prop}
\begin{proof}
A similar calculation to that of Lemma~\ref{lem:zeta_integral}, though decidedly less pleasing. 
Since $\q(\lambda)=0$ for $\lambda \le 1$, we may change the domain of integration to $[1,\infty)$. 
Then use the identity 
\[
\q'(\lambda) = \frac{\q(\lambda)^2(1-\q(\lambda))}{\q(\lambda)+(1-\q(\lambda))\log(1-\q(\lambda))}\, ,
\]
which follows from the fact that $1-\q(\lambda)=e^{-\lambda \q(\lambda)}$ by differentiation. The integral under consideration thus equals
\[
\int_1^{\infty} (1-\q^2(\lambda))\log(1-\q^2(\lambda)) \cdot  \frac{\q(\lambda)+(1-\q(\lambda))\log(1-\q(\lambda))}{\q(\lambda)^2(1-\q(\lambda))}\cdot \q'(\lambda) \mathrm{d}\lambda\, ,
\]
from which the substitution $\q=\q(\lambda)$ gives 
\[
\int_0^{1} \frac{(1+\q)\log(1-\q^2)(\q+(1-\q)\log(1-\q))}{\q^2} \mathrm{d} \q\, .
\]
Substituting $u=-\log(1-\q)$, we have $1-\q=e^{-u}$, $1+\q=2-e^{-u}$ and $\log(1-\q^2)=\log(2-e^{-u})-u$, and the above integral becomes
\[
\int_0^1 \frac{(2-e^{-u}) (\log(2-e^{-u})-u) \cdot (1-e^{-u}-ue^{-u})}{(1-e^{-u})^2}\cdot e^{-u}\mathrm{d}u\, .
\]
This integral can be calculated with a little effort (or easily, for those who accept computer assisted proofs), and equals 
\[
\frac{\pi^2}{3}-6+4\log 2-2 \log ^2(2)
\]
Comparing with (\ref{zmczeta:zmcint_result}) completes the proof (recall that $\zeta(2)=\pi^2/6$). 
\end{proof}

The next lemma generalizes Lemma~\ref{lem:susc_exp_conc}, at the cost of obtaining a non-explicit error bound. We use a slightly different proof technique than for Lemma~\ref{lem:susc_exp_conc}, which exploits that a binomial random variable is reasonably likely to take values close to its mean (see the following exercise). 
\bsex\label{ex:binom_lower_bd}
Show that 
\[
\p{\mathrm{Bin}\left({n \choose 2},\frac{2m}{n^2}\right)=m} = \Omega\left(\frac{1}{n}\right)\, 
\]
uniformly in $0 \le m \le n^2/4$, in that 
\[
\liminf_{n \to \infty} \inf_{m \in \{0,1,\ldots,\lfloor n^2/4\rfloor\}} n \cdot \p{\mathrm{Bin}\left({n \choose 2},\frac{2m}{n^2}\right)=m} > 0. 
\]
\esex
\begin{lem}\label{lem:single_exp_term_approx}
Let $f:[0,1] \to \R$ be continuous. Then 
\[
\limsup_{n \to \infty} \sup_{m \in [{n \choose 2}]} \left|\E{f(\chi(G^{(n)}_m)/n)} - f(\alpha(2m/n)^2)\right| = 0. 
\]
\end{lem}
\begin{proof}
Write $\|f\|=\sup_{x \in [0,1]}f(x) < \infty$.
First suppose $m \ge n^2/4$. Then $\alpha(2m/n) \ge \alpha(n/2) \ge 1-2e^{-n/2}$ by Exercise~\ref{ex:astar_info}~(b), so we have $|f(\alpha^2(2m/n))-f(1)|=o(1)$. 
Next, since $\chi(G)=|G|$ whenever $G$ is connected, by Exercise~\ref{ex:connect_threshold}~(b) we have 
\[
|\E{f(\chi(G^{(n)}_m)/n)} - f(1)| \le \|f\| \p{\chi(G^{(n)}_m)\ne n} = o(1)\, . 
\]
This handles the case $m \ge n^2/4$, so we now assume $0 \le m \le n^2/4$. 

Let $p=2m/n^2$, so $np=2m/n$. By Exercise~\ref{ex:gnp_cond_dist}, we have 
\[
\E{f(\chi(G^{(n)}_m)/n)} = \E{f(\chi(G(n,p))/n)~|~|E(G(n,p))|=m}. 
\]
By Exercise~\ref{ex:binom_lower_bd}, there is $C > 0$ such that for all $m \le n^2/4$, 
\begin{align*}
& \p{|\chi(G(n,p))- n\alpha(2m/n)^2| > 22n^{4/5}~|~|E(G(n,p))|=m} \\
\le & Cn \p{|\chi(G(n,p))- n\alpha(2m/n)^2| > 22n^{4/5}}\\
=& o(1)\, ,
\end{align*}
the last line by Theorem~\ref{thm:susc_conc}. It follows that 
\begin{align*}
& \E{f(\chi(G^{(n)}_m)/n)} 	\\
 \ge & \inf_{|a - n\alpha(2m/n)^2| \le 22n^{4/5}} \E{f(a/n)\I{\chi(G(n,p))=a}~|~|E(G(n,p))|=m}  \\
	&  - \|f\| \cdot \p{|\chi(G(n,p))- n\alpha(2m/n)^2| > 22n^{4/5}~|~|E(G(n,p))|=m} \\
= & \inf_{|a - n\alpha(2m/n)^2| \le 22n^{4/5}} f(a/n) - o(1) \\
= & f(\alpha(2m/n)^2)-o(1)\, ; 
\end{align*}
this bound is uniform in $0 \le m \le n^2/4$ since $f$ is continuous and so uniformly continuous on $[0,1]$. 
We likewise have $\E{f(\chi(G^{(n)}_m)/n)} \le f(\alpha(2m/n)^2) + o(1)$. 
\end{proof}
In what follows we only apply the preceding lemma with $m=o({n \choose 2})$, but it seems more satisfying to prove the estimate over the full range of possibilities; handling larger $m$ only added a few lines to the proof. We are now ready to wrap things up. 
\begin{proof}[Proof of Theorem~~\ref{thm:zmc_lower}]
Write $f(x) = (1-\alpha(x)^2) \log(1-\alpha(x)^2)$ for $x \in [0,1)$ and $f(1)=0$. Then $f$ is continuous, is smooth except at $x=1$, and has $\lim_{x \to \infty} f'(x)=0$ and $\lim_{x \downarrow 1} f'(x)=0$. (To see this, use the defining identity for $\alpha$ to find an identity for $f'$, then use the estimates for $\alpha(x)$  from Exercise~\ref{ex:astar_info}.) Let $C = \sup_{x \ne 1} f'(x) < \infty$. 
Since $\int_0^{\infty}(1-\alpha^2(\lambda))\log(1-\alpha^2(\lambda))\mathrm{d}\lambda < \infty$ and the integrand is negative, we have $\lim_{x \to \infty} \int_x^{\infty}(1-\alpha^2(\lambda))\log(1-\alpha^2(\lambda))\mathrm{d}\lambda=0$. 
It follows as in the proof of Theorem~\ref{thm:frieze} that 
\begin{align*}
& \frac{2}{n} \sum_{m=1}^{5n\log n} (1-\alpha^2(2m/n))\log(1-\alpha^2(2m/n)) \\
= & \int_0^{10\log n}(1-\alpha^2(\lambda))\log(1-\alpha^2(\lambda))\mathrm{d}\lambda + O\left(\frac{\log n}{n}\right)\, \\
= & 2(\zeta_{\textsc{mc}}+\log 2)+o(1)\, ,
\end{align*}
where $\zeta_{\textsc{mc}}$ is defined in (\ref{zmczeta:zmcint_result}). Recalling that $\Xi$, from (\ref{eq:xidef}), is the sum we aim to estimate, by Lemma~\ref{lem:single_exp_term_approx} we then have 
\begin{align*}
\Xi &\le  
 \sum_{m=0}^{5n\log n} \E{\log \pran{1-\frac{\chi(G^{n}_{m})}{n}} \cdot \pran{1-\frac{\chi(G^{(n)}_{m})}{n}}}\\
 & = (\zeta_{\textsc{mc}}+\log 2)\cdot n(1+o(1))\, 
\end{align*}
(recall that $x\log x$ is negative on $[0,1]$). 

If $G^{(n)}_m$ is connected then the $m$'th term in the sum $\Xi$ is zero. Since $\inf_{x \in [0,1]} f(x)=-1/e$, it follows by Lemma~\ref{lem:single_exp_term_approx} and Exercise~\ref{ex:connect_threshold}~(b) that 
\begin{align*}
\Xi &\ge   
(1+\frac{12\log n}{n}) \sum_{m=0}^{5n\log n} \E{\log \pran{1-\frac{\chi(G^{n}_{m})}{n}} \cdot \pran{1-\frac{\chi(G^{(n)}_{m})}{n}}}\\
 &\quad - \frac{1}{e} {n \choose 2} \p{G^{(n)}_{5\log n} \mbox{ is not connected}}\, \\
&  = (\zeta_{\textsc{mc}}+\log 2)\cdot n(1+o(1))\, .
\end{align*}
Applying Proposition~\ref{prop:empirical_mc_identity}, the two preceding inequalities yield 
\[
\E{\log\hat{Z}_{\textsc{mc}}(n)} = \E{\log\hat{Z}_{\textsc{mc}}^{\to}(n)}-(n-1)\log 2 = n\cdot (2\log n + \zeta^{\to}_{\textsc{mc}} + o(1))\, ,
\]
which is the assertion of the theorem. 
\end{proof}
\section{Unanswered questions}\label{sec:openproblems}
The partition function of the multiplicative coalescent provides an interesting avenue by which to approach the probabilistic study of the process. It connects up nicely with other perspectives, and offers its own insights and challenges. We saw above that the empirical partition function of the multiplicative coalescent is a subtle and interesting random variable. Here are a few questions related to $\hat{Z}_{\textsc{mc}}(n)$, and more generally to the multiplicative coalescent, that occurred to me in the course of writing these notes and which I believe deserve investigation. 
\begin{itemize}
\item 
The large deviations of 
$\log\hat{Z}_{\textsc{mc}}(n)$ should be interestingly non-trivial. Can a large deviations rate function be derived? This should be related to large deviations for component sizes in the random graph process. Such results exist for fixed $p$ \citep{oconnell98large,biskup07large}, but not (so far as I am aware) for the process as $p$ varies. (Considering the following sort of problem would be a step in the right direction. Let $E_n$ be the event that the largest component of $G(n,p)$ has at least $0.1n$ fewer vertices than average, for all $p \in [2n,3n]$. Find a law of large numbers for $\log \p{E_n}$.) 

\item Relatedly, what partition chains are responsible for the large value of $\E{\hat{Z}_{\textsc{mc}}(n)}$? It is not too hard to show the following: to maximize $\prod_{i=1}^{n-1} n^2(1-\chi(F_i)^2/n)$ one should keep the component sizes as small as possible. In particular, if $n=2^p$ then one maximizes 
this product by first pairing all singletons to form trees of size two, then pairing these trees to form trees of size $4$, etcetera. This shows that for $n=2^p$, 
\[
\mathrm{ess}\sup \hat{Z}_{\textsc{mc}}(n)  = 2^{-(n-1)}\prod_{k=1}^{p=1} \prod_{j=0}^{n/2^k-1} \left(n^2 - 2^{k-1}(n+j\cdot 2^k)\right)\, .
\]
which is within a factor $4$ of $2^{-(n-1)}n^{2(n-1)}e^{-\log_2 n}$. On the other hand, a straightforward calculation shows the probability of choosing two minimal trees to pair at every step is around $e^{-(1+o(1))2n}$, so the contribution to $\e\hat{Z}_{\textsc{mc}}(n)$ from such paths is $n^{2n}e^{-(1+o(1))(2+\log 2)n}$. This is exponentially small compared to $n^{n-2}(n-1)!$, so the lion's share of the expected value comes from elsewhere. 

\item Suppose we condition $\hat{Z}_{\textsc{mc}}(n)$ to be close to $n^{n-2}(n-1)!=\E{\hat{Z}_{\textsc{mc}}(n)}$; we know by Corollary~\ref{cor:exp_decay} that this event has exponentially small probability. Perhaps, under this conditioning, the tree $T^{(n)}_1$ built by the multiplicative coalescent might be similar to that built by the additive coalescent? At any rate, it would certainly be interesting to study, e.g., $\E{\mathrm{height}(T^{(n)}_1)~|~\hat{Z}_{\textsc{mc}}(n)\ge n^{n-2}(n-1)!}$, or more generally to study observables of $T^{(n)}_1$ under unlikely conditionings of $\hat{Z}_{\textsc{mc}}(n)$. 

\item Condition $T^{(n)}_1$ to have exactly $k$ leaves. After rescaling distances appropriately, $T^{(n)}_1$ should converge in the Gromov-Hausdorff sense. What is the limit? Write $\mathbf{E}_k$ for the coresponding conditional expectation; then we should have, for example, $\mathbf{E}_k[\mathrm{diam}(T^{(n)}_1)/n] \to f(k)$ for some function $f(k)$. How does $f$ behave as $k \to \infty$? It is known \citep{addario09critical} that without conditioning, $\E{\mathrm{diam}(T^{(n)}_1}=\Theta(n^{1/3})$. 

\item Pitman's coalescent, Kingman's coalescent, and the multiplicative coalescent correspond to gelation kernels $\kappa(x,y)=x+y$, $\kappa(x,y)=1$, and $\kappa(x,y)=xy$, respectively. Are there further gelation kernels that may be naturally enriched to form interesting forest-valued coalescent processes? 
\end{itemize}



%
%

\end{document}